\documentclass[11pt]{article}%
\usepackage{amsfonts}
\usepackage{amsmath}
\usepackage{amssymb}
\usepackage{graphicx}
\usepackage{epsfig}
\usepackage{enumerate}
\usepackage{color}
\usepackage{hyperref}

\providecommand{\U}[1]{\protect\rule{.1in}{.1in}}
\newtheorem{theorem}{Theorem}

\newtheorem{definition}[theorem]{Definition}

\newtheorem{lemma}[theorem]{Lemma}

\newtheorem{remark}[theorem]{Remark}

\newenvironment{proof}[1][Proof]{\noindent\textbf{#1.} }{\ \rule{0.5em}{0.5em}}

\renewcommand{\thefootnote}{\fnsymbol{footnote}}
\begin{document}

\author{Rainer Buckdahn$\;^{a,b}$, Tianyang Nie$\;^{a,c,\ast,1}$\bigskip\\{\small ~$^{a}$School of Mathematics, Shandong University, Jinan, Shandong
250100, China}\\{\small $^{b}$Laboratoire de Math\'{e}matiques, Universit\'{e} de Bretagne Occidentale,}\\{\small 29285 Brest C\'{e}dex 3, France}\\{\small ~$^{c}$School of Mathematics and Statistics, University of Sydney, NSW 2006, Australia}}
\title{Generalized Hamilton-Jacobi-Bellman equations with Dirichlet boundary and stochastic exit time optimal control problem}
\maketitle
\date{}

\begin{abstract}
We consider a kind of stochastic exit time optimal control problems, in which the cost function is defined through a nonlinear backward stochastic differential equation. We study the regularity of the value function for such a control problem. Then extending Peng's backward semigroup method, we show the dynamic programming principle. Moreover, we prove that the value function is  a viscosity solution to the following generalized Hamilton-Jacobi-Bellman equation with Dirichlet boundary:
\[
\left\{
\begin{array}
[c]{l}
\inf\limits_{v\in V}\left\{\mathcal{L}(x,v)u(x)+f(x,u(x),\nabla u(x) \sigma(x,v),v)\right\}=0, \quad x\in D,\medskip\\
u(x)=g(x),\quad x\in \partial D,
\end{array}
\right.
\]
where $D$ is a bounded set in $\mathbb{R}^{d}$, $V$ is a compact metric space in $\mathbb{R}^{k}$, and for $u\in C^{2}(D)$ and $(x,v)\in D\times V$,
\[\mathcal{L}(x,v)u(x):=\frac{1}{2}\sum_{i,j=1}^{d}(\sigma\sigma^{\ast})_{i,j}(x,v)\frac{\partial^{2}u}{\partial x_{i}\partial x_{j}}(x)
+\sum_{i=1}^{d}b_{i}(x,v)\frac{\partial u}{\partial x_{i}}(x).
\]
\end{abstract}
\footnotetext[1]{{\scriptsize Corresponding author.}}
\renewcommand{\thefootnote}{\arabic{footnote}}
\footnotetext[1]{{\scriptsize The work of this author is supported under Australian Research Council's Discovery Projects funding scheme (DP120100895)}}
\renewcommand{\thefootnote}{\fnsymbol{footnote}}
\footnotetext{\textit{{\scriptsize E-mail addresses:}}
{\scriptsize rainer.buckdahn@univ-brest.fr (Rainer\ Buckdahn); nietianyang@163.com (Tianyang\ NIE)}}

\medskip


\section{Introduction}
Crandall and Lions introduced the notion of viscosity solution for first order partial differential equations (PDEs) in \cite{CL-83}, and then it was extended to second order PDEs by Lions \cite{L-1983}. In the later work \cite{CIL-92} Crandall et al. gave a systematic investigation of this notion. Viscosity solution provides a powerful tool to study second order PDEs and related problems.

It is by now well known that the classical Hamilton-Jacobi-Bellman (HJB) equation is connected to stochastic optimal control problem, see, e.g. \cite{FR-1975, K-1980}. The reader is referred to \cite{YZ-1999} for a systematic theory of HJB  equations and stochastic control.
For generalized HJB equations as
\[
\left\{
\begin{array}
[c]{l}
\frac{\partial u}{\partial t}+\inf\limits_{v\in V}\left\{\mathcal{L}(x,v)u+f(x,u,\nabla u \sigma(x,v),v)\right\}=0, \quad (t,x)\in (0,T)\times \mathbb{R}^{d},\medskip\\
u(T,x)=g(x),
\end{array}
\right.
\]
Peng \cite{P-1992} was the first to give a stochastic interpretation of the solution to above HJB equation; he did it by investigating a certain optimal control problem in which the cost function is described by a nonlinear backward stochastic differential equation (BSDE) based on the pioneering work of Pardoux and Peng \cite{PP-90}. Moreover, Peng \cite{P-1992} established the dynamic programming principle for the control problem and proved that the value function  is a viscosity solution to above generalized HJB equation. The results were extended by Peng \cite{P-97} with the help of the notion of backward semigroup. The reader is referred to \cite{BBP-97,BL-2008,MY-19992,N-2014,N-20142,PP-92} for further research. Recently, Dumitrescu et al. \cite{DQS-2014} studied combined optimal stopping and stochastic control problems with $\mathcal{E}^f$-expectations defined through BSDEs with jumps, and they investigated their connection with an obstacle problem for an HJB equation. Let us point out that the approach in \cite{DQS-2014} is different from Peng's method and allows the authors to prove, in the case when the reward terminal function is only Borelian, a weak dynamic programming principle.

Motivated by \cite{P-1992, P-97}, we study the following HJB equation with Dirichlet boundary:
\begin{equation}\label{HJB equation in introduction}
\left\{
\begin{array}
[c]{l}
\inf\limits_{v\in V}\left\{\mathcal{L}(x,v)u(x)+f(x,u(x),\nabla u(x) \sigma(x,v),v)\right\}=0, \quad x\in D,\medskip\\
u(x)=g(x),\quad x\in \partial D,
\end{array}
\right.
\end{equation}
where $D$ is a bounded set in $\mathbb{R}^{d}$. In particular, if $f=f(x,v)$, equation (\ref{HJB equation in introduction}) reduces to the Dirichlet problem for the  HJB equation  studied, for example, by Lions and Menaldi \cite{LM-1982}. In \cite{LM-1982}, it was shown that the optimal cost of a control problem belongs to $W^{1,\infty}(D)$ and it is the maximum solution of the HJB equation with Dirichlet boundary. For further research, the reader is referred to \cite{EF-1979, L-19791, L-19792}.

In this paper, we extend the results of \cite{LM-1982} to give a stochastic representation for the viscosity solution of the HJB equation (\ref{HJB equation in introduction}). To do this, we investigate the following stochastic exit time optimal control problem: Consider the stochastic differential equation (SDE)
\[
\left\{
\begin{array}
[c]{l}
dX_{s}^{0,x,v}=b(X_{s}^{0,x,v},v_{s})ds+\sigma(X_{s}^{0,x,v},v_{s})dB_{s}, \quad s\ge 0, \medskip\\
X_{0}^{0,x,v}=x\in\mathbb{R}^{d},
\end{array}
\right.
\]
where $B$ is an $\mathbb{R}^{m}$-valued Brownian motion, $b$ and $\sigma$ are given functions satisfying suitable assumptions, and $v=\{v_{s}\}$ is an admissible control taking values in a compact metric space $V\in\mathbb{R}^{k}$. Let $D$ be a bounded set of $\mathbb{R}^{d}$ and $\tau_{x,v}$ be the first exit time of $X^{0,x,v}$ from $\overline{D}$. To define our cost function, we introduce the nonlinear BSDE with random terminal time:
\[
Y_{t}^{0,x,v}=g(X_{\tau_{x,v}}^{0,x,v})+\int_{t\wedge\tau_{x,v}}^{\tau_{x,v}}f(X_{s}^{0,x,v},Y_{s}^{0,x,v},Z_{s}^{0,x,v},v_{s})ds
-\int_{t\wedge\tau_{x,v}}^{\tau_{x,v}}Z_{s}^{0,x,v}dB_{s},
\]
where $f$ and $g$ are given functions defined on $\mathbb{R}^{d}\times\mathbb{R}\times\mathbb{R}^{m}\times V$ and $\mathbb{R}^{d}$, respectively. The well-posedness of above BSDE was established first by Peng \cite{P-1991} and later extended by Darling and Pardoux \cite{DP-1997}; see also \cite{BH-1998,P-1998,R-2004}. Now we define the cost function $J(x,v):=Y_{0}^{0,x,v}$ and the value function $u(x):=\inf\limits_{v}J(x,v)$ for our stochastic exit time optimal control problem.

Our objective is to prove that the value function $u$ defined above is the viscosity solution of the HJB equation (\ref{HJB equation in introduction}). The first step is to show some regularity results for $u$. Let us first recall the results for the case $f=f(s,x,v)$. In general, when $D$ is bounded,  the continuity of $u$ is not always true, see \cite{KD-2001} page 278-279. Fleming and Soner \cite{FS-2006} found a sufficient conditions such that $u$ is continuous (see Theorem 2.1 \cite{FS-2006}) and Bayraktar et al. \cite{BSY-2011} weakened the assumptions of  \cite{FS-2006}. If $f=f(x,v)$ and $\sigma$ is non-degenerate, under some suitable assumptions on $D$, the Lipschitz continuity of $u$ was obtained by Lions and Menaldi \cite{LM-1982}. They also extended the results to the degenerate case in \cite{LM-19822}. We mention that the results of \cite{LM-1982} were generalized by \cite{BB-1995,BCI-2008,IL-1990,K-2009}  under weaker assumptions. In this paper, motivated by \cite{LM-1982},  we prove for non-degenerate $\sigma$, that our value function $u$ defined above is $\frac{1}{2}$-H\"{o}lder continuous. Since our value function is defined through a nonlinear BSDE with random terminal time, it is more general than that in \cite{LM-1982}. To show the regularity, we need the stability property of BSDE w.r.t. the perturbations, see the proof of our Theorem \ref{regularity of u}. Instead of the Lipschitz continuity as in \cite{LM-1982}, we get in our framework the $\frac{1}{2}$-H\"{o}lder continuity of $u$.

In a second step we study the dynamic programming principle (DPP). As by now well known, for $f=f(s,x,v)$, the DPP holds, see e.g. \cite{FS-2006} and \cite{MY-1999}.  For a cost function defined by a BSDE with deterministic terminal time, the DPP was first shown by Peng \cite{P-1992}. Then it was proven again by Peng \cite{P-97} using the method of backward semigroup. We emphasise that we cannot just follow the procedure of \cite{P-97} to prove the DPP for our value function $u$, because the terminal time of our BSDE (see (\ref{BSDE coupled with SDE})) is  the stochastic exit time of SDE (\ref{SDE}). This stochastic exit time depends not only on the initial date $x$ but also on the control process $v\in\mathcal{V}$. We have to establish the following relation (see Lemma \ref{Lemma for equivalent form of value function})
\[
u(x)=\inf\limits_{v\in\mathcal{V}}Y_{0}^{0,x,v}=\mbox{essinf}_{v\in\mathcal{V}}Y_{\Theta}^{\Theta,x,v},
\]
which is not obviously at all. To prove this, we introduce the time-shift operator and make a subtle analysis. For more details, see Section 4. With the help of above relation and Peng's backward semigroup method, we can show that the DPP is also satisfied, see Theorem \ref{DPP}.

In Section 5, using the regularity property of the value function $u$ and the dynamic programming principle, we can show that $u$ is the viscosity solution of the HJB equation (\ref{HJB equation in introduction}). We emphasise that the random terminal time makes the application of the procedure of Peng \cite{P-97} more complicate, and so we need a special subtle approach, see e.g. Lemma \ref{lemma 3 for proving viscosity solution}.

The paper is organised  as follows: In Section 2 we formulate the problem. We introduce our assumptions and recall existing essential results on BSDE with random terminal time. Section 3 is devoted to the study of the value function, and in particular, its regularity.  In Section 4 the dynamic programming principle is established. Section 5 is devoted to the proof that the function $u$ is a viscosity solution of the HJB equation (\ref{HJB equation in introduction}) and we also have the uniqueness of the viscosity solution for such HJB equation.

\section{Formulation of the problem}
Let $(\Omega,\mathcal{F},\mathbb{P})$ be the classical Wiener space: $\Omega:=C_{0}(\mathbb{R}_{+};\mathbb{R}^{m})$ is the set of all continuous functions from $\mathbb{R}_{+}$ to $\mathbb{R}^{m}$ starting from $0$, $\mathcal{F}$ is the Borel $\sigma$-algebra over $\Omega$, completed by the Wiener measure $\mathbb{P}$. In this probability space, the coordinate process $B_{s}(\omega)=\omega(s)$, $s\ge 0$, $\omega\in\Omega$, is an $\mathbb{R}^{m}$-valued Brownian motion. We denote by $\mathbb{F}:=\{\mathcal{F}_{t},t\ge0\}$ the filtration generated by the Brownian motion $B$ and augmented by $\mathcal{N}_{\mathbb{P}}$ (the class of $\mathbb{P}$-null sets of $\mathcal{F}$).

Through the paper, for $d,m\ge1$, we use the notations $|x|^{2}:=\sum\limits_{i=1}^{d}x_{i}^{2}$, for $x\in\mathbb{R}^{d}$, and $|A|^{2}:=\sum\limits_{i=1}^{d}\sum\limits_{j=1}^{m}a_{ij}^{2}$, for $A\in \mathbb{R}^{d\times m}$.

For $x\in\mathbb{R}^{d}$, we consider the following SDE with control:
\begin{equation}\label{SDE}
\left\{
\begin{array}
[c]{l}
dX_{s}^{0,x,v}=b(X_{s}^{0,x,v},v_{s})ds+\sigma(X_{s}^{0,x,v},v_{s})dB_{s}, \quad s\ge 0,\medskip\\
X_{0}^{0,x,v}=x\in\mathbb{R}^{d},
\end{array}
\right.
\end{equation}
where $v=\{v_{s}\}$ is an $\{\mathcal{F}_{s}\}$-adapted process taking its values in a compact set $V\subset \mathbb{R}^{k}$. The coefficients $b:\mathbb{R}^{d}\times V\rightarrow \mathbb{R}^{d}$ and $\sigma: \mathbb{R}^{d}\times V\rightarrow \mathbb{R}^{d\times m}$ are supposed to be continuous and to satisfy the following assumptions:
\begin{itemize}
\item[$\left(  H_{1}\right)  $] There exists a positive constant $L$ such that for all $x,x_{1},x_{2}\in\mathbb{R}^{d}$, $v\in V$,
\[
\begin{array}
[c]{rl}
(i) & |b(x_{1},v)-b(x_{2},v)|+|\sigma(x_{1},v)-\sigma(x_{2},v)|\leq
L|x_{1}-x_{2}|,\medskip\\
(ii)& |b(x,v)|+|\sigma(x,v)|\leq L(1+|x|).
\end{array}
\]
\end{itemize}

We denote by $\mathcal{V}$ the set of admissible control processes composed of all $V$-valued $\{\mathcal{F}_{s}\}$-progressively measurable processes. Then we know that under assumption $(H_{1})$, equation (\ref{SDE}) has a unique strong solution for each given $v\in\mathcal{V}$.

Let $D\subset\mathbb{R}^{d}$ be a bounded domain. For each $(x,v)\in D\times \mathcal{V}$, we define the first exit time $\tau_{x,v}$ of $X^{0,x,v}$ from the bounded domain $\overline{D}$:
\begin{equation}\label{exit time}
\tau_{x,v}:=\inf\{t\ge0: X_{t}^{0,x,v}\notin\overline{D}\}.
\end{equation}
From the right continuity of $\{\mathcal{F}_{s}\}$, we know that $\tau_{x,v}$ is a stopping time w.r.t. $\{\mathcal{F}_{s}\}$, see, e.g. Dyknin \cite{D-1965}.

Given $(x,v)\in\mathbb{R}^{d}\times \mathcal{V}$, let us consider the nonlinear BSDE with random terminal time :
\begin{equation}\label{BSDE coupled with SDE}
Y_{t}^{0,x,v}=g(X_{\tau_{x,v}}^{0,x,v})+\int_{t\wedge\tau_{x,v}}^{\tau_{x,v}}f(X_{s}^{0,x,v},Y_{s}^{0,x,v},Z_{s}^{0,x,v},v_{s})ds
-\int_{t\wedge\tau_{x,v}}^{\tau_{x,v}}Z_{s}^{0,x,v}dB_{s},
\end{equation}
where $f$ and $g$ are given functions satisfying the following assumptions:
\begin{itemize}

\item[$\left(H_{2}\right)$] The function $g: \mathbb{R}^{d}\rightarrow \mathbb{R}$ is continuous.

\item[$\left(H_{3}\right)$] $f:\mathbb{R}^{d}\times\mathbb{R}\times\mathbb{R}^{m}\times V\rightarrow\mathbb{R}$ is a continuous function which restriction on $\overline{D}\times\mathbb{R}\times\mathbb{R}^{m}\times V$ is such that, for some constants $L\ge0$, $\beta\ge0$ and $\alpha$ (positive or negative), such that, for all $x,x_{1},x_{2}\in \overline{D}$, $y,y_{1},y_{2}\in\mathbb{R}$, $z,z_{1}.z_{2}\in \mathbb{R}^{1\times m}$, $v\in V$,
    \[
    \begin{array}[c]{rl}
      (i) & |f(x,y,z,v)|\leq |f(x,0,z,v)|+L(1+|y|),  \medskip\\
      (ii)& |f(x_{1},y,z_{1},v)-f(x_{2},y,z_{2},v)|\leq \beta(|x_{1}-x_{2}|+|z_{1}-z_{2}|),\medskip\\
      (iii)& (y_{1}-y_{2})(f(x,y_{1},z,v)-f(x,y_{2},z,v))\leq -\alpha|y_{1}-y_{2}|^{2}.
      \end{array}
    \]
\end{itemize}
\begin{remark}\label{remark for boundedness of the coefficient}
From $(H_{1})$-$(H_{3})$ it follows easily  that the functions $b,\sigma,g$ and $f(\cdot,0,0,\cdot)$ are bounded in $\overline{D}\times V$.
\end{remark}

In addition to $(H_{1})$-$(H_{3})$ we need some technical assumptions:
\begin{itemize}

\item[$\left(H_{4}^{\prime}\right)$] For each $v\in\mathcal{V}$, the set of regular points
$\Gamma:=\left\{x\in\partial D: \mathbb{P}(\tau_{x,v}>0)=0\right\}$
is closed. Moreover, there exists some $\mu\in\mathbb{R}$, such that $\sup\limits_{x\in\overline{D},v\in \mathcal{V}}E[\exp(\mu\tau_{x,v})]<\infty.$
\item[$\left(H_{5}\right)$] For $\mu$ introduced in $(H_{4}^{\prime})$, we assume that $\mu>\gamma:=\beta^{2}-2\alpha$.
\end{itemize}

In our paper, we focus on the case that $\sigma$ is non-degenerate and $D$ satisfies a uniform exterior sphere condition, which means
\begin{itemize}
\item[$\left(H_{4}\right)$] (1) (Non-degeneracy) There exists a real number $\lambda>0$, s.t.
\[
\sum_{i,j=1}^{d}\left(\sigma\sigma^{\ast}(x,v)\right)_{ij}a_{i}a_{j}\ge \lambda|a|^{2}, \text{ for all } a\in\mathbb{R}^{d},\ x\in \overline{D } \text{ and } v\in V.
\]
(2)(Uniform exterior sphere condition) There exists a constant $\rho>0$, such that
\[
\text{ for all } y\in \partial D, \text{ there exists } \tilde{y}\in \mathbb{R}^{d}\setminus D, \text{ s.t. }
\overline{D}\cap\{z\in\mathbb{R}^{d}:|\tilde{y}-z|\leq \rho\}=\{y\}.
\]
\end{itemize}
\begin{remark}
Using the results of Khasminskii \cite{K-2012} or Lions and Menaldi \cite{LM-1982}, we know that $(H_{4})$ is stronger than $(H_{4}^{\prime})$. Indeed, $(H_{4})$ implies the existence of a positive $\mu$ such that $(H_{4}^{\prime})$ holds. For the readers' convenience, we give details in next section.
\end{remark}
Now we apply the results of Darling and Pardoux; see Theorem 3.4 \cite{DP-1997}, or Lemma \ref{Wellposedness Darling and Pardoux} below (For the readers' convenience, we recall some results of \cite{DP-1997} at the end of this section). Considering Remark \ref{remark for boundedness of the coefficient}, we have
\begin{theorem}\label{theorem for wellposedness of BSDE coupled with SDE}
Suppose $(H_{1})$-$(H_{5})$ (i.e. also $(H_4)$). Then, for each $x\in D$ and $v\in\mathcal{V}$, BSDE (\ref{BSDE coupled with SDE}) has a unique solution $(Y^{0,x,v},Z^{0,x,v})\in M_{\gamma}^{2}(0,\tau_{x,v};\mathbb{R})\times M_{\gamma}^{2}(0,\tau_{x,v};\mathbb{R}^{m})$.
Moreover, the solution belongs to
$M_{\mu}^{2}(0,\tau_{x,v};\mathbb{R})\times M_{\mu}^{2}(0,\tau_{x,v};\mathbb{R}^{m})$ and
$E[\sup\limits_{0\leq s\leq \tau_{x,v}}e^{\mu s}|Y^{0,x,v}_{s}|^{2}]<\infty$.
Here, for any real number $\theta$, any stopping time $\tau$, and any Euclidean space $U$, $M_{\theta}^{2}(0,\tau;U)$ denotes the Hilbert space of progressively measurable processes $\{\eta(s)\}$ s.t.
\[
\|\eta\|_{\theta}^{2}=E\left[\int_{0}^{\tau}e^{\theta s}|\eta(s)|^{2}ds\right]<\infty.
\]
\end{theorem}

Now let us introduce the cost function of our stochastic exit time control problem. Motivated by Peng \cite{P-1991,P-1992}, we define our recursive cost functional on $\mathbb{R}^{d}\times \mathcal{V}$ as
\[
J(x,v):=Y_{0}^{0,x,v}=E\left[g(X_{\tau_{x,v}}^{0,x,v})+\int_{0}^{\tau_{x,v}}f(X_{s}^{0,x,v},Y_{s}^{0,x,v},Z_{s}^{0,x,v},v_{s})ds\right],\quad
(x,v)\in\mathbb{R}^{d}\times \mathcal{V},
\]
where $(Y^{0,x,v},Z^{0,x,v})$ is the unique solution of BSDE (\ref{BSDE coupled with SDE}),
and we introduce the value function as
\begin{equation}\label{value function}
u(x):=\inf\limits_{v\in\mathcal{V}}J(x,v)=\inf\limits_{v\in\mathcal{V}}Y_{0}^{0,x,v},\quad x\in\mathbb{R}^{d}.
\end{equation}

One of our main objectives is to show that the value function $u$ defined above is a viscosity solution of the following generalised Hamilton-Jacobi-Bellman equation with Dirichlet boundary:
\[
\left\{
\begin{array}
[c]{l}
\inf\limits_{v\in V}\left\{\mathcal{L}(x,v)u(x)+f(x,u(x),\nabla u(x) \sigma(x,v)),v\right\}=0, \quad x\in D,\medskip\\
u(x)=g(x),\quad x\in \partial D,
\end{array}
\right.
\]
where, for $u\in C^{2}(D)$ and $(x,v)\in D\times V$,
\[
\mathcal{L}(x,v)u(x):=\frac{1}{2}\sum_{i,j=1}^{d}(\sigma\sigma^{\ast})_{i,j}(x,v)\frac{\partial^{2}u}{\partial x_{i}\partial x_{j}}(x)
+\sum_{i=1}^{d}b_{i}(x,v)\frac{\partial u}{\partial x_{i}}(x).
\]
For this end, we will first investigate the regularity of $u$; see Section 3.

Finally, at the end of this section, we recall some essential results of \cite{DP-1997}. Let us first recall the following well-posedness results for BSDEs with random terminal time; see Theorem 3.4 \cite{DP-1997}:
\begin{lemma}\label{Wellposedness Darling and Pardoux}
Let $\tau$ be an $\{\mathcal{F}_{s}\}$-stopping time and $\xi$ be an $\mathcal{F}_{\tau}$-measurable random variable in $\mathbb{R}^{n}$. Let $h:\Omega\times\mathbb{R}_{+}\times\mathbb{R}^{n}\times\mathbb{R}^{n\times m}\rightarrow\mathbb{R}^{n}$ be a function satisfying the following assumptions:
\begin{itemize}
\item[$\left(A_{1}\right)$] The map $y\mapsto h(s,y,z)$ is continuous. There exist constants $L\ge0$, $\beta\ge0$, $\alpha$ (positive or negative) s.t., for all $y,y_{1},y_{2}\in\mathbb{R}$, $z,z_{1}.z_{2}\in \mathbb{R}^{n\times m}$, a.s.,
    \[
    \begin{array}[c]{rl}
      (i) & |h(s,y,z)|\leq |h(s,0,z)|+L(1+|y|), \medskip\\
      (ii)& |h(s,y,z_{1})-h(s,y,z_{2})|\leq \beta|z_{1}-z_{2}|,\medskip\\
      (iii)& \langle y_{1}-y_{2},h(s,y_{1},z)-h(s,y_{2},z)\rangle\leq -\alpha|y_{1}-y_{2}|^{2}.
      \end{array}
    \]
\end{itemize}
We also assume that, for some $\mu>\gamma=\beta^{2}-2\alpha$,
\[
E\left[e^{\mu\tau}(|\xi|^{2}+1)+\int_{0}^{\tau}e^{\mu s}|h(s,0,0)|^{2}ds\right]<\infty.
\]
Then there exists a unique solution $(Y,Z)\in M_{\gamma}^{2}(0,\tau;\mathbb{R}^{n})\times M_{\gamma}^{2}(0,\tau;\mathbb{R}^{n\times m})$ of the BSDE:
\begin{equation}\label{general BSDE with random terminal time}
Y_{t}=\xi+\int_{t\wedge\tau}^{\tau}h(s,Y_{s},Z_{s})ds
-\int_{t\wedge\tau}^{\tau}Z_{s}dB_{s}, \quad t\ge0.
\end{equation}
Moreover, this solution belongs to
$M_{\mu}^{2}(0,\tau;\mathbb{R}^{n})\times M_{\mu}^{2}(0,\tau;\mathbb{R}^{n\times m})$, and
$E[\sup\limits_{0\leq s\leq \tau}e^{\mu s}|Y_{s}|^{2}]<\infty$.
\end{lemma}

Let us also recall the stability w.r.t. perturbations and the comparison theorem for BSDEs with random terminal time; see Proposition 4.4 and Corollary  4.4.2 \cite{DP-1997}. Here we adopt the convention that
$Y_{s}=Y_{\tau}=\xi$, $Z_{s}=0$ and $f(s,y,z)=0$ on $\{s>\tau\}$.

\begin{lemma}\label{Stability w.r.t. perturbation}
Suppose the triples $(\tau,\xi,h)$ and $(\tau^{\prime},\xi^{\prime},h^{\prime})$ satisfy the conditions in Lemma \ref{Wellposedness Darling and Pardoux} with the same $\alpha$, $\beta$ and $\mu>\beta^{2}-2\alpha$. Then, for all $\beta^{2}-2\alpha<\theta\leq\mu$, for the unique solution $(Y,Z)\in M_{\mu}^{2}(0,\tau;\mathbb{R}^{n})\times M_{\mu}^{2}(0,\tau;\mathbb{R}^{n\times m})$ (resp., $(Y^{\prime},Z^{\prime})\in M_{\mu}^{2}(0,\tau^{\prime};\mathbb{R}^{n})\times M_{\mu}^{2}(0,\tau^{\prime};\mathbb{R}^{n\times m})$) of BSDE (\ref{general BSDE with random terminal time}) related to $(\tau,\xi,h)$ (resp., $(\tau^{\prime},\xi^{\prime},h^{\prime})$), if we denote $\Delta Y=Y-Y^{\prime}$ and $\Delta Z=Z-Z^{\prime}$, we have that
\[
\begin{array}[c]{ll}
|\Delta Y(0)|^{2}+C_{1}E\left[\displaystyle\int_{0}^{\tau\vee\tau^{\prime}}e^{\theta s}\left(|\Delta Y(s)|^{2}+|\Delta Z(s)|^{2}\right)ds\right] \medskip\\
\leq E\left[\left|\xi e^{\theta\tau/2}-\xi^{\prime}e^{\theta\tau^{\prime}/2}\right|^{2}\right]
+C_{2}^{-1}E\left[\displaystyle\int_{0}^{\tau\vee\tau^{\prime}}e^{\theta s}\left|h(s,Y(s),Z(s))-h^{\prime}(s,Y(s),Z(s))\right|^{2}ds\right].
\end{array}
\]
Here $C_1,C_2>0$ are constant depending on the constants introduced in assumption $(A_1)$.
\end{lemma}

\begin{lemma}\label{comparison theorem}
Under the assumptions of Lemma \ref{Stability w.r.t. perturbation}, for the case $n=1$, $\tau=\tau^{\prime}$, $h\leq h^{\prime}$ and $\xi\leq \xi^{\prime}$, we have $Y(t)\leq Y^{\prime}(t)$, a.s.
\end{lemma}
\section{Regularity of value function}
We begin with the following lemma; see Lions and Menaldi \cite{LM-1982}. For the convenience of the reader, considering the importance of the lemma, we give its proof here.
\begin{lemma}\label{lemma for assumption}
Under assumption $(H_{1})$, if we have $(H_{4})$, then there exists a positive constant $\mu$ such that $(H_{4}^{\prime})$ holds.
\end{lemma}
\begin{proof}
First, from Corollary 3.3 \cite{K-2012}  or  Lemma 2.4 \cite{LM-1982} we know that under the assumptions $(H_{1})$ and $(H_{4})$, there exists a constant $\mu>0$ such that $\sup_{x\in\overline{D},v\in \mathcal{V}}Ee^{\mu \tau_{x,v}}<\infty$.

To prove $(H_{4}^{\prime})$, we also have to show that $\Gamma:=\left\{x\in\partial D: \mathbb{P}(\tau_{x,v}>0)=0\right\}$ is closed. We claim that we even have $\Gamma=\partial D$. Indeed, for any fixed $y\in\partial D$, due to $(H_{4})$, there exists $\tilde{y}\in \mathbb{R}^{d}/D$, s.t. $\overline{D}\cap\{z:|\tilde{y}-z|\leq \rho\}=\{y\}$. Now we introduce the function $w(x,y):=e^{-k\rho^{2}}-e^{-k|x-\tilde{y}|^{2}}$, $x\in \overline{D}$,  for some $k>0$. It's not hard to check that, for $1\leq i,j\leq d$,
\[
\frac{\partial w}{\partial x_{i}}(x,y):=2k(x_{i}-\tilde{y}_{i})e^{-k|x-\tilde{y}|^{2}}, ~
\frac{\partial^{2}w}{\partial x_{i}\partial x_{j}}(x,y):=2ke^{-k|x-\tilde{y}|^{2}}\delta_{i,j}-4k^{2}(x_{i}-\tilde{y}_{i})(x_{j}-\tilde{y}_{j})e^{-k|x-\tilde{y}|^{2}},
\]
and
\[
\begin{array}[c]{rl}
\mathcal{L}(x,v)w(x,y)
=& e^{-k|x-\tilde{y}|^{2}}\Big(-2k^{2}\sum_{i,j=1}^{d}(\sigma\sigma^{\ast})_{i,j}(x,v)(x_{i}-\tilde{y}_{i})(x_{j}-\tilde{y}_{j})\medskip\\
& +k\sum_{i=1}^{d}(\sigma\sigma^{\ast})_{i,i}(x,v)
+2k\sum_{i=1}^{d}b_{i}(x,v)(x_{i}-\tilde{y}_{i})\Big).
\end{array}
\]
From the assumptions $(H_{1})$ and $(H_{4})$, the boundedness of $D$ and $|x-\tilde{y}|\ge \rho>0$, $x\in \overline{D}$, it follows that, for $k$ large enough, there exists a strictly positive constant $\tilde{\mu}$, s.t.
\[
-\mathcal{L}(x,v)w(x,y)\ge \tilde{\mu}, \text{ for all } x\in \overline{D}.
\]
Applying It\^{o}'s formula to $w(X^{0,y,v}_{s},y)$ and taking the expectation, we obtain
\[
0\leq E\left[w(X^{0,y,v}_{t\wedge\tau_{y,v}},y)\right]\leq w(y,y)-E\left[\int_{0}^{t\wedge\tau_{y,v}}\tilde{\mu} ds\right]
\]
and, thus, $E[\tilde{\mu}(t\wedge\tau_{y,v})]\leq w(y,y)=0$.
Hence, from Fatou's lemma we have $E[\tilde{\mu}\tau_{y,v}]=0$.
Therefore, $\mathbb{P}(\tau_{y,v}=0)=1$ and $y\in\Gamma$.
\end{proof}
\smallskip

Let us recall the definition of the value function
\[
\begin{array}[c]{rl}
u(x):=\inf\limits_{v\in\mathcal{V}}J(x,v)=\inf\limits_{v\in\mathcal{V}}Y_{0}^{0,x,v}
=
\inf\limits_{v\in\mathcal{V}}E\left[g(X_{\tau_{x,v}}^{0,x,v})+\displaystyle\int_{0}^{\tau_{x,v}}f(X_{s}^{0,x,v},Y_{s}^{0,x,v},Z_{s}^{0,x,v},v_{s})ds\right].
\end{array}
\]
In the following part of this section, we will show that $u$ is $1/2$-H\"{o}lder continuous.
Before doing this, we present two auxiliary lemmas.
\begin{lemma}\label{estimate for forward SDE}
Under the assumptions $(H_{1})$ and $(H_{4})$, for any real-valued stopping time $\tilde{\theta}$, we have
\[
E\left[|X^{0,x,v}_{\tilde{\theta}}-X^{0,x^{\prime},v}_{\tilde{\theta}}|^{2}e^{-2\delta\tilde{\theta}}\right]\leq |x-x^{\prime}|^{2},
\quad x,x^{\prime}\in\mathbb{R}^{d},~v\in\mathcal{V},
\]
where
\[
\delta:=\sup_{x,x^{\prime}\in \mathbb{R}^{d}, v\in V}\left\{\frac{1}{2}\text{Tr}\frac{(\sigma(x,v)-\sigma(x^{\prime},v))(\sigma(x,v)-\sigma(x^{\prime},v))^{\ast}}{|x-x^{\prime}|^{2}}+
\frac{(x-x^{\prime})\cdot(b(x,v)-b(x^{\prime},v))}{|x-x^{\prime}|^{2}}\right\}.
\]
\end{lemma}
\begin{proof}
We apply It\^{o}'s formula to $|X_{s}^{0,x,v}-X_{s}^{0,x^{\prime},v}|^{2}e^{-2\delta s}$ between $0$ and $\tilde{\theta}\wedge t$. It follows that
\[
\begin{array}[c]{l}
\quad E\left[|X_{\tilde{\theta}\wedge t}^{0,x,v}-X_{\tilde{\theta}\wedge t}^{0,x^{\prime},v}|^{2}e^{-2\delta (\tilde{\theta}\wedge t)}\right]\medskip\\
=|x-x^{\prime}|^{2}+E\Big[\displaystyle\int_{0}^{\tilde{\theta}\wedge t}\big\{\text{Tr}(\sigma(X_{r}^{0,x,v},v)-\sigma(X_{r}^{0,x^{\prime},v},v))(\sigma(X_{r}^{0,x,v},v)-\sigma(X_{r}^{0,x^{\prime},v},v))^{\ast}\medskip\\
\qquad\qquad\quad+2(X_{r}^{0,x,v}-X_{r}^{0,x^{\prime},v})\cdot(b(X_{r}^{0,x,v},v)-b(X_{r}^{0,x^{\prime},v},v))
-2\delta|X_{r}^{0,x,v}-X_{r}^{0,x^{\prime},v}|^{2}\big\}e^{-2\delta r}dr\Big].
\end{array}
\]
Thus, from the definition of $\delta$, we have
$E\left[|X_{\tilde{\theta}\wedge t}^{0,x,v}-X_{\tilde{\theta}\wedge t}^{0,x^{\prime},v}|^{2}e^{-2\delta (\tilde{\theta}\wedge t)}\right]
\leq|x-x^{\prime}|^{2}$,
and letting $t\rightarrow\infty$, we obtain from Fatou's lemma and the continuity of $X_{r}^{0,x,v}$ in $r$ that
\[
E\left[|X_{\tilde{\theta}}^{0,x,v}-X_{\tilde{\theta}}^{0,x^{\prime},v}|^{2}e^{-2\delta \tilde{\theta}}\right]
\leq|x-x^{\prime}|^{2}.
\]
The proof is complete.
\end{proof}

Now we consider the function $w$ introduced in the proof of Lemma \ref{lemma for assumption}. Given $y\in\partial D$, we let $\tilde{y}\in\mathbb{R}^{d}\backslash D$ be the element for which $\overline{D}\cap \{z\in\mathbb{R}^{d}:|\tilde{y}-z|\leq \rho\}=\{y\}$ (see $(H_{4})$). For $x\in\overline{D}$ and $k>0$ we define as before $w(x,y):=e^{-k\rho^{2}}-e^{-k|x-\tilde{y}|^{2}}$. Let $w(x):=\inf_{y\in \partial D}w(x,y)$, $x\in\overline{D}$. Then $w\in W^{1,\infty}(D)$, $w\ge0$ and $w=0$ on $\partial D$. In particular, we have the following lemma,
\begin{lemma}\label{lemma for lipschitz estimate}
We suppose $(H_{1})$-$(H_{5})$. We also assume that there exists a constant $\theta$ such that $\beta^{2}-2\alpha<\theta\leq\mu$ and $\theta\leq -2[\delta]^{+}$. Then there exists a constant $\mu_{0}>0$, such that
\[
E[e^{\theta(\tau_{x^{\prime},v}\wedge\tau_{x,v})/2}-e^{\theta\tau_{x,v}/2}]\leq
\frac{|\theta|}{2\mu_{0}}\|\nabla w\|_{\infty}|x-x^{\prime}|,\quad x,~x^{\prime}\in\overline{D}.
\]
\end{lemma}
\begin{proof}
We observe that for $\theta=0$, the lemma holds obviously. Now it is sufficient to consider the case that $\theta\leq-2[\delta]^{+}$ and $\theta<0$. Recall that $w(x,y):=e^{-k\rho^{2}}-e^{-k|x-\tilde{y}|^{2}}$, $x\in \overline{D}$, $k>0$, where $\tilde{y}$ associated with $y$ by $(H_{4})$. Similarly to the proof of Lemma \ref{lemma for assumption}, for any fixed $\theta$, we know that for $k$ large enough, there exists a constant $\mu_{0}>0$, s.t.
\[
-\mathcal{L}(x,v)w(x,y)-\frac{\theta}{2} w(x,y)\ge \mu_{0}, \text{ for all } x\in \overline{D}.
\]
We apply It\^{o}'s formula to $w(X_{s\wedge\tau_{x,v}}^{0,x,v},y)e^{\theta (s\wedge\tau_{x,v})/2}$ and take the conditional expectation. Then
\begin{equation}\label{supermartingale inequality}
\begin{array}[c]{l}
\quad E[\mu_{0}\displaystyle\int_{0}^{t\wedge\tau_{x,v}}e^{\theta r/2}dr+w(X^{0,x,v}_{t\wedge\tau_{x,v}},y)e^{\theta (t\wedge\tau_{x,v})/2}|\mathcal{F}_{s\wedge\tau_{x,v}}] \medskip\\
=\mu_{0}\displaystyle\int_{0}^{s\wedge\tau_{x,v}}e^{\theta r/2}dr+w(X^{0,x,v}_{s\wedge\tau_{x,v}},y)e^{\theta (s\wedge\tau_{x,v})/2}\medskip\\
\qquad\qquad\qquad+E[\displaystyle\int_{s\wedge\tau_{x,v}}^{t\wedge\tau_{x,v}}(\mathcal{L}(X^{x,v}_{r},v_{r})w(X^{0,x,v}_{r},y)+\frac{\theta}{2} w(X^{0,x,v}_{r},y)+\mu_{0})e^{\theta r/2}dr|\mathcal{F}_{s\wedge\tau_{x,v}}]\medskip\\
\leq \mu_{0}\displaystyle\int_{0}^{s\wedge\tau_{x,v}}e^{\theta r/2}dr+w(X^{0,x,v}_{s\wedge\tau_{x,v}},y)e^{\theta (s\wedge\tau_{x,v})/2}, \quad t\ge s.
\end{array}
\end{equation}
This means that $\mu_{0}\int_{0}^{t\wedge\tau_{x,v}}e^{\theta r/2}dr+w(X^{0,x,v}_{t\wedge\tau_{x,v}},y)e^{\theta (t\wedge\tau_{x,v})/2}$, $t\ge0$, is a supermartingale, continuous and  bounded on bounded time interval.

Recall that $w(x):=\inf_{y\in \partial D}w(x,y)$, $x\in\overline{D}$. Obviously, there is some $\mathcal{F}_{s\wedge\tau_{x,v}}$-measurable random variable $\xi$, such that $w(X^{0,x,v}_{s\wedge\tau_{x,v}})=w(X^{0,x,v}_{s\wedge\tau_{x,v}},\xi)$. Then from (\ref{supermartingale inequality}) it follows
\[
\begin{array}[c]{rl}
& \mu_{0}\displaystyle\int_{0}^{s\wedge\tau_{x,v}}e^{\theta r/2}dr+w(X^{0,x,v}_{s\wedge\tau_{x,v}})e^{\theta (s\wedge\tau_{x,v})/2}\medskip\\
\ge & E[\mu_{0}\displaystyle\int_{0}^{t\wedge\tau_{x,v}}e^{\theta r/2}dr+w(X^{0,x,v}_{t\wedge\tau_{x,v}},\xi)e^{\theta (t\wedge\tau_{x,v})/2}|\mathcal{F}_{s\wedge\tau_{x,v}}] \medskip\\
\ge & E[\mu_{0}\displaystyle\int_{0}^{t\wedge\tau_{x,v}}e^{\theta r/2}dr+w(X^{0,x,v}_{t\wedge\tau_{x,v}})e^{\theta (t\wedge\tau_{x,v})/2}|\mathcal{F}_{s\wedge\tau_{x,v}}], \quad \mathbb{P}-a.s.
\end{array}
\]
This shows that also $\mu_{0}\int_{0}^{t\wedge\tau_{x,v}}e^{\theta r/2}dr+w(X^{0,x,v}_{t\wedge\tau_{x,v}})e^{\theta (t\wedge\tau_{x,v})/2}$, $t\ge0$, is a supermartingale; it is also continuous and bounded on bounded time interval. Therefore, from Doob's optional stopping theorem, it follows that, for $x, x^{\prime}\in \overline{D}$
\[
\begin{array}[c]{rl}
& E\left[\mu_{0}\displaystyle\int_{0}^{t\wedge\tau_{x,v}}e^{\theta r/2}dr+w(X^{0,x,v}_{t\wedge\tau_{x,v}})e^{\theta(t\wedge\tau_{x,v})/2}\big|\mathcal{F}_{t\wedge\tau_{x^{\prime},v}\wedge\tau_{x,v}}\right]\medskip\\
\leq & \mu_{0}\displaystyle\int_{0}^{t\wedge\tau_{x^{\prime},v}\wedge\tau_{x,v}}e^{\theta r/2}dr+w(X^{0,x,v}_{t\wedge\tau_{x^{\prime},v}\wedge\tau_{x,v}})e^{\theta (\tau_{x^{\prime},v}\wedge\tau_{x,v})/2}, \quad \mathbb{P}-a.s., \quad t\ge 0.
\end{array}
\]
Taking the expectation on both sides and the limit as $t\rightarrow\infty$ (Recall that $\tau_{x^{\prime},v}$ and $\tau_{x,v}$ are finite, $\mathbb{P}$-a.s.), we get from the monotone convergence theorem
\[
\begin{array}[c]{rl}
E[\mu_{0}\displaystyle\int^{\tau_{x,v}}_{\tau_{x^{\prime},v}\wedge\tau_{x,v}}e^{\theta r/2}dr]\leq
E[w(X^{0,x,v}_{\tau_{x^{\prime},v}\wedge\tau_{x,v}})e^{\theta (\tau_{x^{\prime},v}\wedge\tau_{x,v})/2}
-w(X^{0,x,v}_{\tau_{x,v}})e^{\theta\tau_{x,v}/2}].
\end{array}
\]
Using the definition of $\tau_{x,v}$, we have $w(X^{0,x,v}_{\tau_{x,v}})=w(X^{0,x^{\prime},v}_{\tau_{x^{\prime},v}})=0\leq w(X^{0,x^{\prime},v}_{\tau_{x^{\prime},v}\wedge\tau_{x,v}})$. Thus,
\[
\begin{array}[c]{rl}
E[\mu_{0}\displaystyle\int^{\tau_{x,v}}_{\tau_{x^{\prime},v}\wedge\tau_{x,v}}e^{\theta r/2}dr]
&\leq
E[w(X^{0,x,v}_{\tau_{x^{\prime},v}\wedge\tau_{x,v}})e^{\theta (\tau_{x^{\prime},v}\wedge\tau_{x,v})/2}
-w(X^{0,x,v}_{\tau_{x,v}})e^{\theta\tau_{x,v}/2}]\medskip\\
& = E[(w(X^{0,x,v}_{\tau_{x^{\prime},v}\wedge\tau_{x,v}})
-w(X^{0,x^{\prime},v}_{\tau_{x^{\prime},v}}))e^{\theta (\tau_{x^{\prime},v}\wedge\tau_{x,v})/2}]\medskip\\
&=E[1_{\{\tau_{x^{\prime},v}\ge\tau_{x,v}\}}(w(X^{0,x,v}_{\tau_{x^{\prime},v}\wedge\tau_{x,v}})
-w(X^{0,x^{\prime},v}_{\tau_{x^{\prime},v}})e^{\theta (\tau_{x^{\prime},v}\wedge\tau_{x,v})/2}]\medskip\\
&\quad+E[1_{\{\tau_{x^{\prime},v}<\tau_{x,v}\}}(w(X^{0,x,v}_{\tau_{x^{\prime},v}\wedge\tau_{x,v}})
-w(X^{0,x^{\prime},v}_{\tau_{x^{\prime},v}}))e^{\theta (\tau_{x^{\prime},v}\wedge\tau_{x,v})/2}]\medskip\\
& =E[1_{\{\tau_{x^{\prime},v}<\tau_{x,v}\}}(w(X^{0,x,v}_{\tau_{x^{\prime},v}\wedge\tau_{x,v}})
-w(X^{0,x^{\prime},v}_{\tau_{x^{\prime},v}\wedge\tau_{x,v}}))e^{\theta (\tau_{x^{\prime},v}\wedge\tau_{x,v})/2}]\medskip\\
& \leq E[|w(X^{0,x,v}_{\tau_{x^{\prime},v}\wedge\tau_{x,v}})
-w(X^{0,x^{\prime},v}_{\tau_{x^{\prime},v}\wedge\tau_{x,v}})|e^{\theta (\tau_{x^{\prime},v}\wedge\tau_{x,v})/2}]\medskip\\
&\leq \|\nabla w\|_{\infty}E[|X^{0,x,v}_{\tau_{x^{\prime},v}\wedge\tau_{x,v}}-X^{0,x^{\prime},v}_{\tau_{x^{\prime},v}\wedge\tau_{x,v}}|e^{\theta (\tau_{x^{\prime},v}\wedge\tau_{x,v})/2}],
\end{array}
\]
where $\|\cdot\|_{\infty}:=\|\cdot\|_{L^\infty}$ denotes the $L^\infty$-norm over $\overline{D}$.
From Lemma \ref{estimate for forward SDE} and $\theta\leq-2[\delta]^{+}$ we have
\[
\begin{array}[c]{rl}
&E[|X^{0,x,v}_{\tau_{x^{\prime},v}\wedge\tau_{x,v}}-X^{0,x^{\prime},v}_{\tau_{x^{\prime},v}\wedge\tau_{x,v}}|e^{\theta (\tau_{x^{\prime},v}\wedge\tau_{x,v})/2}]\medskip\\
\leq &
\left\{E[|X^{0,x,v}_{\tau_{x^{\prime},v}\wedge\tau_{x,v}}-X^{0,x^{\prime},v}_{\tau_{x^{\prime},v}\wedge\tau_{x,v}}|^{2}e^{\theta (\tau_{x^{\prime},v}\wedge\tau_{x,v})}]\right\}^{1/2}
\leq|x-x^{\prime}|.
\end{array}
\]
Consequently, as $\theta<0$, it follows
$2\frac{\mu_{0}}{|\theta|}E[e^{\theta(\tau_{x^{\prime},v}\wedge\tau_{x,v})/2}-e^{\theta\tau_{x,v}/2}]\leq
\|\nabla w\|_{\infty}|x-x^{\prime}|$,
from which we obtain
\[
E[e^{\theta(\tau_{x^{\prime},v}\wedge\tau_{x,v})/2}-e^{\theta\tau_{x,v}/2}]\leq
\frac{|\theta|}{2\mu_{0}}\|\nabla w\|_{\infty}|x-x^{\prime}|.
\]
\hfill
\end{proof}

\noindent Now we can give the following theorem to characterise the regularity of value function $u(x)$.
\begin{theorem}\label{regularity of u}
We suppose that the assumptions $(H_{1})$-$(H_{5})$ are satisfied. We also assume that $g\in W^{2,\infty}(D)$ and there exists a constant $\theta$ such that $\beta^{2}-2\alpha<\theta\leq\mu$ and $\theta< -2[\delta]^{+}$. Then there exists a constant $C$, such that for all $x,x^{\prime}\in \overline{D}$, we have
\[
|u(x)-u(x^{\prime})|\leq \sup_{v\in\mathcal{V}}|Y_{0}^{0,x,v}-Y_{0}^{0,x^{\prime},v}|\leq C|x-x^{\prime}|^{1/2}.
\]
\end{theorem}
\begin{proof}
Applying Lemma \ref{Stability w.r.t. perturbation} and recalling (\ref{value function}), we have
\[
|u(x)-u(x^{\prime})|^{2}\leq \sup_{v\in\mathcal{V}}|Y_{0}^{0,x,v}-Y_{0}^{0,x^{\prime},v}|^{2}\leq I_{1}+I_{2},
\]
where, for $\beta^{2}-2\alpha<\theta\leq\mu$,
$I_{1}:=\sup_{v\in\mathcal{V}}E\left[|e^{\frac{\theta}{2}\tau_{x,v}}g(X^{0,x,v}_{\tau_{x,v}})-
e^{\frac{\theta}{2}\tau_{x^{\prime},v}}g(X^{0,x^{\prime},v}_{\tau_{x^{\prime},v}})|^{2}\right]$
and
\[
I_{2}:=\sup_{v\in\mathcal{V}}C_{2}^{-1}
E[\int_{0}^{\tau_{x,v}\vee \tau_{x^{\prime},v}}e^{\theta r}|f(X_{r}^{0,x,v},Y_{r}^{0,x,v},Z_{r}^{0,x,v},v_{r})-
f(X_{r}^{0,x^{\prime},v},Y_{r}^{0,x,v},Z_{r}^{0,x,v},v_{r})|^{2}dr].
\]
Since $g\in W^{2,\infty}(D)$, using It\^{o}'s formula for Sobolev spaces (See, e.g. Chapter 2, Section 10 in Krylov \cite{K-1980}), it follows that (Notice that $\theta\neq 0$, since $\theta< -2[\delta]^{+}$)

\[
\begin{array}[c]{rl}
I_{1}&=\sup_{v\in\mathcal{V}}E\left[|e^{\frac{\theta}{2}\tau_{x,v}}g(X^{0,x,v}_{\tau_{x,v}})-
e^{\frac{\theta}{2}\tau_{x^{\prime},v}}g(X^{0,x^{\prime},v}_{\tau_{x^{\prime},v}})|^{2}\right]\medskip\\
&\leq \sup_{v\in\mathcal{V}}\Big\{3E\left[|e^{\frac{\theta}{2}\tau_{x,v}}g(X^{0,x,v}_{\tau_{x,v}})
-e^{\frac{\theta}{2}(\tau_{x,v}\wedge\tau_{x^{\prime},v})}g(X^{0,x,v}_{\tau_{x,v}\wedge \tau_{x^{\prime},v}})|^{2}\right]\medskip\\
&\quad+3E\left[|e^{\frac{\theta}{2}\tau_{x^{\prime},v}}g(X^{0,x^{\prime},v}_{\tau_{x^{\prime},v}})
-e^{\frac{\theta}{2}(\tau_{x,v}\wedge\tau_{x^{\prime},v})}g(X^{0,x^{\prime},v}_{\tau_{x,v}\wedge \tau_{x^{\prime},v}})|^{2}\right]\medskip\\
&\quad+3E\left[e^{\theta(\tau_{x,v}\wedge\tau_{x^{\prime},v})}|g(X^{0,x,v}_{\tau_{x,v}\wedge \tau_{x^{\prime},v}})-g(X^{0,x^{\prime},v}_{\tau_{x,v}\wedge \tau_{x^{\prime},v}})|^{2}\right]\Big\}\medskip\\
&\leq 2\sup_{v\in \mathcal{V}}
\Big\{\frac{12}{|\theta|^2}
E\left[|e^{\frac{\theta}{2}\tau_{x,v}}-e^{\frac{\theta}{2} (\tau_{x,v}\wedge\tau_{x^{\prime},v})}|^{2}\right]
\sup_{v\in V}\left(\|\mathcal{L}(\cdot,v)g(\cdot)+\frac{\theta}{2}g(\cdot)\|_{\infty}^{2}\right)\medskip\\
&\quad+\frac{3}{|\theta|}E\left[|e^{\theta\tau_{x,v}}-e^{\theta(\tau_{x,v}\wedge\tau_{x^{\prime},v})}|\right]
\sup_{v\in V}\left(\|\nabla g(\cdot)\sigma(\cdot,v)\|_{\infty}^{2}\right)\medskip\\
&\quad+\frac{12}{|\theta|^2}E\left[|e^{\frac{\theta}{2}\tau_{x^{\prime},v}}-e^{\frac{\theta}{2} (\tau_{x,v}\wedge\tau_{x^{\prime},v})}|^{2}\right]
\sup_{v\in V}\left(\|\mathcal{L}(\cdot,v)g(\cdot)+\frac{\theta}{2}g(\cdot)\|_{\infty}^{2}\right)\medskip\\
&\quad+\frac{3}{|\theta|}E\left[|e^{\theta\tau_{x^{\prime},v}}-e^{\theta(\tau_{x,v}\wedge\tau_{x^{\prime},v})}|\right]
\sup_{v\in V}\left(\|\nabla g(\cdot)\sigma(\cdot,v)\|_{\infty}^{2}\right)\medskip\\
&\quad+3E\left[e^{\theta(\tau_{x,v}\wedge\tau_{x^{\prime},v})}|X^{0,x,v}_{\tau_{x,v}\wedge \tau_{x^{\prime},v}}-X^{0,x^{\prime},v}_{\tau_{x,v}\wedge \tau_{x^{\prime},v}}|^{2}\right]\|\nabla g\|_{\infty}^{2}\Big\}.
\end{array}
\]
Recall that from Lemma \ref{estimate for forward SDE} we have
\[
\begin{array}[c]{l}
Ee^{\theta(\tau_{x,v}\wedge\tau_{x^{\prime},v})}|X^{0,x,v}_{\tau_{x,v}\wedge \tau_{x^{\prime},v}}-X^{0,x^{\prime},v}_{\tau_{x,v}\wedge \tau_{x^{\prime},v}}|^{2}
\leq|x-x^{\prime}|^{2}.
\end{array}
\]

Now we are going to estimate $E|e^{\frac{\theta}{2}\tau_{x,v}}-e^{\frac{\theta}{2}(\tau_{x,v}\wedge\tau_{x^{\prime},v})}|^{2}$ and $E|e^{\theta\tau_{x,v}}-e^{\theta(\tau_{x,v}\wedge\tau_{x^{\prime},v})}|$. Using Lemma \ref{lemma for lipschitz estimate}, it's not hard to obtain that there exist constants $C>0$, independent of $x,x^{\prime},v$, such that
\[
E\left[|e^{\frac{\theta}{2}\tau_{x,v}}-e^{\frac{\theta}{2}(\tau_{x,v}\wedge\tau_{x^{\prime},v})}|^{2}\right]
\leq 2E\left[|e^{\frac{\theta}{2}\tau_{x,v}}-e^{\frac{\theta}{2}(\tau_{x,v}\wedge\tau_{x^{\prime},v})}|\right]
\leq C|x-x^{\prime}|,
\]
and
\[
\begin{array}[c]{rl}
E\left[|e^{\theta\tau_{x,v}}-e^{\theta(\tau_{x,v}\wedge\tau_{x^{\prime},v})}|\right]
&=E\left[|e^{\frac{\theta}{2}\tau_{x,v}}+e^{\frac{\theta}{2}(\tau_{x,v}\wedge\tau_{x^{\prime},v})}|
|e^{\frac{\theta}{2}\tau_{x,v}}-e^{\frac{\theta}{2}(\tau_{x,v}\wedge\tau_{x^{\prime},v})}|\right]\medskip\\
&\leq 2E|e^{\frac{\theta}{2}\tau_{x,v}}-e^{\frac{\theta}{2}(\tau_{x,v}\wedge\tau_{x^{\prime},v})}|\leq C|x-x^{\prime}|.
\end{array}
\]
Finally, let us compute $I_{2}$. For $\theta<-2[\delta]^{+}$, we denote $c_{0}:=-2[\delta]^{+}-\theta>0$. Then, using Lemma \ref{estimate for forward SDE} it follows
\[
\begin{array}[c]{rl}
I_{2}&=\sup\limits_{v\in\mathcal{V}}C_{2}^{-1}
E[\displaystyle\int_{0}^{\tau_{x,v}\vee \tau_{x^{\prime},v}}e^{\theta r}|f(X_{r}^{0,x,v},Y_{r}^{0,x,v},Z_{r}^{0,x,v},v_{r})-
f(X_{r}^{0,x^{\prime},v},Y_{r}^{0,x,v},Z_{r}^{0,x,v},v_{r})|^{2}dr]\medskip\\
&\leq\sup_{v\in\mathcal{V}}C_{2}^{-1}\beta^{2}
E[\displaystyle\int_{0}^{\tau_{x,v}\vee \tau_{x^{\prime},v}}e^{\theta r}|X_{r}^{0,x,v}-X_{r}^{0,x^{\prime},v}|^{2}dr]\medskip\\
&\leq\sup_{v\in\mathcal{V}}C_{2}^{-1}\beta^{2}
E[\displaystyle\int_{0}^{\infty}e^{-c_{0}r}e^{-2[\delta]^{+}(r\wedge(\tau_{x,v}\vee \tau_{x^{\prime},v}))}
|X_{r\wedge(\tau_{x,v}\vee \tau_{x^{\prime},v})}^{0,x,v}
-X_{r\wedge(\tau_{x,v}\vee \tau_{x^{\prime},v})}^{0,x^{\prime},v}|^{2}dr]\medskip\\
&\leq |x-x^{\prime}|^{2}C_{2}^{-1}\beta^{2}\displaystyle\int_{0}^{\infty}e^{-c_{0}r}dr\leq \frac{\beta^{2}}{C_{2}c_{0}}|x-x^{\prime}|^{2}.
\end{array}
\]
Therefore, there exists a constant $C>0$ such that, for all $x,x^{\prime}\in\overline{D}$,
\[
|u(x)-u(x^{\prime})|\leq C|x-x^{\prime}|^{1/2}.
\]
\hfill
\end{proof}
\begin{remark}
Let us point out that we can follow the approach of \cite{LM-1982} to show the regularity of $u$. However, the method of \cite{LM-1982} needs the boundedness of $f$, and translating this method to our framework, we cannot show that $u$ is Lipschitz continuous, but only $1/2$-H\"{o}lder continuous.
\end{remark}

\section{Dynamic programming principle}
In this section, we will establish the dynamic programming principle (DPP) for our stochastic exit time optimal control problem. The main idea is to extend the stochastic backward semigroup introduced by Peng \cite{P-97} to BSDEs with random terminal time.

For $(x,v)\in \mathbb{R}^{d}\times \mathcal{V}$, we recall SDE (\ref{SDE}) and the definition of the exit time $\tau_{x,v}$ (see (\ref{exit time})). Then, for a given stopping time $\Theta$ and a real valued $\mathcal{F}_{\tau_{x,v}\wedge \Theta}$-measurable random variable $\eta$ satisfying $E[e^{\mu \tau_{x,v}}|\eta|^{2}]<+\infty$,  we know from Lemma \ref{Wellposedness Darling and Pardoux} that the following BSDE
\[
\tilde{Y}_{t}^{0,x,v}=\eta+\int_{t\wedge\tau_{x,v}\wedge\Theta}^{\tau_{x,v}\wedge\Theta}f(s,X_{s}^{0,x,v},\tilde{Y}_{s}^{0,x,v},\tilde{Z}_{s}^{0,x,v},v_s)ds
-\int_{t\wedge\tau_{x,v}\wedge\Theta}^{\tau_{x,v}\wedge\Theta}\tilde{Z}_{s}^{0,x,v}dB_{s}, \quad t\ge0,
\]
has a unique solution $(\tilde{Y}^{0,x,v},\tilde{Z}^{0,x,v})\in M_{\gamma}^{2}(0,\tau_{x,v}\wedge\Theta;\mathbb{R})\times M_{\gamma}^{2}(0,\tau_{x,v}\wedge\Theta;\mathbb{R}^{m})$. Moreover, this solution belongs to
$M_{\mu}^{2}(0,\tau_{x,v}\wedge\Theta;\mathbb{R})\times M_{\mu}^{2}(0,\tau_{x,v}\wedge\Theta;\mathbb{R}^{m})$ and we also have
\[
E[\sup\limits_{0\leq s\leq \tau_{x,v}\wedge\Theta}e^{\mu s}|\tilde{Y}_{s}^{0,x,v}|^{2}]<\infty.
\]
We define the backward semigroup by setting
$G_{s,\tau_{x,v}\wedge\Theta}^{0,x,v}[\eta]:=\tilde{Y}_{s\wedge\tau_{x,v}}^{0,x,v}$,
and for simplicity we denote
$G_{\tau_{x,v}\wedge\Theta}^{0,x,v}[\eta]:=\tilde{Y}_{0}^{0,x,v}$.
Then obviously, for the solution $(Y^{0,x,v},Z^{0,x,v})$ of BSDE (\ref{BSDE coupled with SDE}), we have
\begin{equation}\label{backward semigroup}
Y_{0}^{0,x,v}=G_{\tau_{x,v}}^{0,x,v}[g(X_{\tau_{x,v}}^{0,x,v})]=G_{\tau_{x,v}\wedge\Theta}^{0,x,v}[Y_{\tau_{x,v}\wedge\Theta}^{0,x,v}]=
G_{\tau_{x,v}\wedge\Theta}^{0,x,v}[Y_{\Theta}^{0,x,v}],
\end{equation}
since $Y_{\tau_{x,v}\wedge\Theta}^{0,x,v}=Y_{\Theta}^{0,x,v}$. Now we give the main result of this section.
\begin{theorem}\label{DPP}
We suppose $(H_{1})$-$(H_{5})$ are satisfied. We also assume that $g\in W^{2,\infty}(D)$ and the existence of a constant $\theta$ such that $\beta^{2}-2\alpha<\theta\leq\mu$ and $\theta< -2[\delta]^{+}$. Then, for any stopping time $\Theta$ such that $Ee^{\mu \Theta}<\infty$, we have
\[
u(x)=\inf\limits_{v\in\mathcal{V}}G_{\tau_{x,v}\wedge\Theta}^{0,x,v}[u(X_{\tau_{x,v}\wedge\Theta}^{0,x,v})].
\]
(Recall that $u(x):=\inf\limits_{v\in\mathcal{V}}Y_{0}^{0,x,v}$; see (\ref{value function})).
\end{theorem}
\begin{proof}
The theorem can be obtained directly from the following Lemmas \ref{lemma 2 for proving DPP} and \ref{lemma 3 for proving DPP}.
\end{proof}

\medskip

To state the Lemmas \ref{lemma 2 for proving DPP} and \ref{lemma 3 for proving DPP}, we have first to establish two results.
For this end, for a given stopping time $\Theta$, we define the time-shift operator $\pi_{\Theta}:\Omega\rightarrow\Omega$,
\[
\pi_{\Theta}(\omega)_{s}:=\omega(\Theta(\omega)+s)-\omega(\Theta(\omega)), ~\omega\in\Omega
\]
(Recall that $\Omega=C_0(\mathbb{R}_{+};\mathbb{R}^m)$).
We also introduce the filtration $\mathcal{F}_{s}^{\Theta}:=\sigma\{B_{r}^{\Theta}:=B_{\Theta+r}-B_{\Theta}, ~0\leq r\leq s\}\vee\mathcal{N}_{\mathbb{P}}$, $s\ge0$, and we denote by $\mathcal{V}_{\Theta}:=L^{0}_{\mathcal{F}^{\Theta}}(0,+\infty;V)$ the set of all $V$-valued $\{\mathcal{F}_{s}^{\Theta}\}$-progressively measurable processes. Then we have, with the identification of $drd\mathbb{P}$-a.e. coinciding processes,
\begin{equation}\label{shift control space and original control space}
\mathcal{V}_{\Theta}=\mathcal{V}(\pi_{\Theta}):=\left\{v(\pi_{\Theta}),~v\in\mathcal{V}\right\}.
\end{equation}
Indeed, on the one hand, for any $v\in \mathcal{V}_{\Theta}$, there exists a non-anticipating measurable function $\widetilde{v}:\mathbb{R}_{+}\times C_{0}(\mathbb{R}_{+};\mathbb{R}^{m})\rightarrow V$, such that $v_{r}=\widetilde{v}(r,B^{\Theta})$, $drd\mathbb{P}$-a.e.
Let $\hat{v}_{r}:=\widetilde{v}(r,B)$, $r\ge0$. Then $\hat{v}\in\mathcal{V}$ and
$\hat{v}(\pi_{\Theta})=\widetilde{v}(\cdot,B^{\Theta})=v$, $drd\mathbb{P}$-a.e.
Thus, with the identification of control processes which coincides $drd\mathbb{P}$-a.e., we have $\mathcal{V}_{\Theta}\subseteq\mathcal{V}(\pi_{\Theta})$. On the other hand, for all $v\in\mathcal{V}$, there exists a non-anticipating measurable function $\widetilde{v}:\mathbb{R}_{+}\times C_{0}(\mathbb{R}_{+};\mathbb{R}^{m})\rightarrow V$, such that $v_{r}=\widetilde{v}(r,B)$,  $drd\mathbb{P}$-a.e., and $v(\pi_{\Theta})=\widetilde{v}(\cdot,B^{\Theta})\in \mathcal{V}_{\Theta}$. This means that  $\mathcal{V}_{\Theta}\supseteq\mathcal{V}(\pi_{\Theta})$. Therefore, (\ref{shift control space and original control space}) is proved.

\begin{lemma}\label{Lemma for equivalent form of value function}
Under the assumptions $(H_{1})$-$(H_{5})$, for a given stopping time $\Theta$ such that $Ee^{\mu \Theta}<\infty$ and for any $\xi\in L^{2}(\mathcal{F}_{\Theta},\overline{D})$ and $v\in \mathcal{V}$,  we consider
\begin{equation}\label{SDE starting from Theta}
X_{t}^{\Theta,\xi,v}=\xi+\int_{\Theta}^{t}b(X_{s}^{\Theta,\xi,v},v_{s})ds
+\int_{\Theta}^{t}\sigma(X_{s}^{\Theta,\xi,v},v_{s})dB_{s}, \quad t\ge \Theta,
\end{equation}
and we define $\tau_{\Theta,\xi,v}=\inf\{t\ge\Theta: X_{t}^{\Theta,\xi,v}\notin \overline{D}\}$. Then we have, for $\xi=x\in \overline{D}$,
\[
u(x)=\inf\limits_{v\in\mathcal{V}}Y_{0}^{0,x,v}=\mbox{essinf}_{v\in\mathcal{V}}Y_{\Theta}^{\Theta,x,v},\quad \mathbb{P}\text{-}a.s., \quad x\in \overline{D},
\]
where $(Y^{\Theta,\xi,v},Z^{\Theta,\xi,v})$ is the solution of the following BSDE, for $t\ge\Theta$,
\begin{equation}\label{BSDE coupled SDE starting from Theta}
Y_{t}^{\Theta,\xi,v}=g(X_{\tau_{\Theta,\xi,v}}^{\Theta,\xi,v})+\int_{t\wedge\tau_{\Theta,\xi,v}}^{\tau_{\Theta,\xi,v}}
f(X_{s}^{\Theta,\xi,v},Y_{s}^{\Theta,\xi,v},Z_{s}^{\Theta,\xi,v},v_{s})ds
-\int_{t\wedge\tau_{\Theta,\xi,v}}^{\tau_{\Theta,\xi,v}}Z_{s}^{\Theta,\xi,v}dB_{s}.
\end{equation}
We will cite (\ref{SDE starting from Theta}) and (\ref{BSDE coupled SDE starting from Theta}) as SDE and BSDE with initial data $(\Theta,\xi)$, respectively.
\end{lemma}
\begin{proof}
Applying an argument similar to that for Lemma \ref{lemma for assumption} and using $Ee^{\mu\Theta}<\infty$, one can show that $Ee^{\mu\tau_{\Theta,x,v}}<\infty$. Then following the proof of Theorem \ref{theorem for wellposedness of BSDE coupled with SDE}, we can show that $Y_{s}^{\Theta,x,v}$ is well defined.

\textbf{Step 1:}
Let us first show that, for $v\in\mathcal{V}$,
\begin{equation}\label{shifted FBSDE and FBSDE}
(X_{t}^{0,x,v},Y_{t}^{0,x,v},Z_{t}^{0,x,v})(\pi_{\Theta})
=(X_{t}^{(\Theta),x,v^{\Theta}},Y_{t}^{(\Theta),x,v^{\Theta}},Z_{t}^{(\Theta),x,v^{\Theta}}),\quad t\ge0,
\end{equation}
where $(X_{t}^{(\Theta),x,v^{\Theta}},Y_{t}^{(\Theta),x,v^{\Theta}},Z_{t}^{(\Theta),x,v^{\Theta}})$ is the unique solution of SDE (\ref{SDE}) and BSDE (\ref{BSDE coupled with SDE}) driven by $B^{\Theta}$ with control $v^{\Theta}=v(\pi_{\Theta})$,  i.e.
\[
\left\{
\begin{array}
[c]{l}
X_{t}^{(\Theta),x,v^{\Theta}}=x+\displaystyle\int_{0}^{t}b(X_{s}^{(\Theta),x,v^{\Theta}},v_{s}^{\Theta})ds
+\displaystyle\int_{0}^{t}\sigma(X_{s}^{(\Theta),x,v^{\Theta}},v_{s}^{\Theta})dB^{\Theta}_{s},\medskip\\
Y_{t}^{(\Theta),x,v^{\Theta}}=g(X_{\tau_{(\Theta),x,v}}^{(\Theta),x,v^{\Theta}})
+\displaystyle\int_{t\wedge\tau_{(\Theta),x,v}}^{\tau_{(\Theta),x,v}}f(X_{s}^{(\Theta),x,v^{\Theta}},
Y_{s}^{(\Theta),x,v^{\Theta}},Z_{s}^{(\Theta),x,v^{\Theta}},v_{s}^{\Theta})ds\medskip\\
\qquad\qquad\qquad\qquad\qquad\qquad\qquad
-\displaystyle\int_{t\wedge\tau_{(\Theta),x,v}}^{\tau_{(\Theta),x,v}}Z_{s}^{(\Theta),x,v^{\Theta}}dB^{\Theta}_{s},\quad t\ge0,
\end{array}
\right.
\]
where $\tau_{(\Theta),x,v^{\Theta}}=\inf\{t\ge0: X_{t}^{(\Theta),x,v^{\Theta}}\notin \overline{D}\}$.
Indeed, as aforementioned, for any given $v\in\mathcal{V}$, there exists a non-anticipating measurable function $\widetilde{v}:\mathbb{R}_{+}\times C_{0}(\mathbb{R}_{+};\mathbb{R}^{m})\rightarrow V$ such that $v=\widetilde{v}(\cdot,B)$, $ds\times d\mathbb{P}$-a.e.
Then comparing both $(X_{t}^{0,x,v},Y_{t}^{0,x,v},Z_{t}^{0,x,v})(\pi_{\Theta})$ and $(X_{t}^{(\Theta),x,v^{\Theta}},Y_{t}^{(\Theta),x,v^{\Theta}},Z_{t}^{(\Theta),x,v^{\Theta}})$, we obtain (\ref{shifted FBSDE and FBSDE}) easily from the uniqueness of the solution to above system of equations. Related with, we get
\begin{equation}\label{shifted exit time and exit time}
\tau_{(\Theta),x,v^{\Theta}}=(\tau_{x,v})(\pi_{\Theta}).
\end{equation}

\textbf{Step 2:}  We recall that we work on the classical Wiener space $(\Omega,\mathcal{F},\mathbb{P})$, where $\Omega:=C_{0}(\mathbb{R}_{+};\mathbb{R}^{m})$, $\mathbb{P}$ is Wiener measure and $\mathcal{F}:=\mathcal{B}(\Omega)\vee \mathcal{N}_{\mathbb{P}}$. Then a given stopping time $\Theta:\Omega\rightarrow\mathbb{R}_{+}$ defines the following canonical decomposition:
\[
(\Omega,\mathcal{F},\mathbb{P})\equiv (\Omega^{\prime},\mathcal{F}^{\prime},\mathbb{P}^{\prime})
\otimes(\Omega^{\prime\prime},\mathcal{F}^{\prime\prime},\mathbb{P}^{\prime\prime}),
\]
where $\Omega^{\prime}=\Omega^{\prime\prime}=\Omega$, $\mathbb{P}^{\prime}:=\mathbb{P}_{B_{\wedge\Theta}}$ ($B_{\wedge\Theta}$ denotes the stopped Brownian motion $B_{t\wedge\Theta}=\omega(t\wedge\Theta(\omega))$, $\omega\in\Omega$), $\mathbb{P}^{\prime\prime}:=\mathbb{P}_{B^{\Theta}}=\mathbb{P}$,  $\mathcal{F}^{\prime}=\mathcal{B}(\Omega)\vee \mathcal{N}_{\mathbb{P}^{\prime}}$ and  $\mathcal{F}^{\prime\prime}=\mathcal{B}(\Omega)\vee \mathcal{N}_{\mathbb{P}^{\prime\prime}}$. For $\omega\in \Omega$, we have $\omega\equiv(\omega^{\prime},\omega^{\prime\prime})\in\Omega^{\prime}\otimes\Omega^{\prime\prime}$, where
$\omega^{\prime}(s)=\omega_{\wedge\Theta}(s):=\omega(s\wedge\Theta(\omega))$ and
$\omega^{\prime\prime}(s)=\omega(\Theta(\omega)+s)-\omega(\Theta(\omega))$, $s\ge0$. This leads for $(\omega^{\prime},\omega^{\prime\prime})\in\Omega^{\prime}\otimes\Omega^{\prime\prime}$ to the identification
$\omega(s)\equiv \omega^{\prime}(s)+\omega^{\prime\prime}((s-\Theta(\omega^{\prime})^{+})$, $s\ge0$.

Recalling now that $v=\widetilde{v}(\cdot,B)$, we set
\[
\hat{v}^{\omega^{\prime}}(s,\omega^{\prime\prime}):=\widetilde{v}(s+\Theta(\omega^{\prime}),\omega^{\prime},\omega^{\prime\prime}), \quad
\omega=(\omega^{\prime},\omega^{\prime\prime}).
\]
We observe that, for all $\omega^{\prime}\in\Omega^{\prime}$, $\hat{v}^{\omega^{\prime}}$ is  a measurable, non-anticipating function over $\mathbb{R}_{+}\times\Omega^{\prime\prime}$. This has, in particular, as consequence that $\hat{v}^{\omega^{\prime}}(\cdot,B)\in \mathcal{V}$.

We claim that for $\mathbb{P}^{\prime}$-almost all $\omega^{\prime}\in\Omega^{\prime}$, $\mathbb{P}$-a.s.
\begin{equation}\label{shifted SDE and SDE}
X_{\Theta+t}^{\Theta,x,v}(\omega^{\prime},\cdot)=X_{t}^{(\Theta),x,\hat{v}^{\omega^{\prime}}(\cdot,B^{\Theta})}
=(X_{t}^{0,x,\hat{v}^{\omega^{\prime}}(\cdot,B)})(\pi_{\Theta(\omega^{\prime})}), \quad t\ge0.
\end{equation}
Indeed, recall that
\[
X_{\Theta+t}^{\Theta,x,v}=x+\int_{\Theta}^{\Theta+t}b(X_{s}^{\Theta,x,v},v_{s})ds
+\int_{\Theta}^{\Theta+t}\sigma(X_{s}^{\Theta,x,v},v_{s})dB_{s}.
\]
Then using $\hat{v}^{\omega^{\prime}}(s,B^{\Theta}(\omega))=\widetilde{v}(s+\Theta(\omega^{\prime}),\omega^{\prime},B^{\Theta}(\omega))$, we have for $\mathbb{P}^{\prime}$-almost all $\omega^{\prime}\in\Omega^{\prime}$, $\mathbb{P}$-a.s.,
\[
\begin{array}[c]{rl}
X_{\Theta+t}^{\Theta,x,v}(\omega^{\prime},\cdot)&
=x+\left(\displaystyle\int_{\Theta}^{\Theta+t}b(X_{s}^{\Theta,x,v},\widetilde{v}(s,B))ds\right)(\omega^{\prime},\cdot)
+\left(\displaystyle\int_{\Theta}^{\Theta+t}\sigma(X_{s}^{\Theta,x,v},\widetilde{v}(s,B))dB_{s}\right)(\omega^{\prime},\cdot),\medskip\\
&=x+\displaystyle\int_{0}^{t}b(X_{\Theta+s}^{\Theta,x,v}(\omega^{\prime},\cdot),\hat{v}^{\omega^{\prime}}(s,B^{\Theta}))ds
+\displaystyle\int_{0}^{t}\sigma(X_{\Theta+s}^{\Theta,x,v}(\omega^{\prime},\cdot),\hat{v}^{\omega^{\prime}}(s,B^{\Theta}))dB_{s}^{\Theta},
\end{array}
\]
$t\ge0$. From the uniqueness of the solution we get
\[
X_{\Theta+t}^{\Theta,x,v}(\omega^{\prime},\cdot)=X_{t}^{(\Theta),x,\hat{v}^{\omega^{\prime}}(\cdot,B^{\Theta})}
\quad t\ge0 \quad \mathbb{P}\text{-}a.s., ~ \mathbb{P}^{\prime}(d\omega^\prime)\text{-}a.s.
\]
Then (\ref{shifted SDE and SDE}) is obtained by combining the above equation with (\ref{shifted FBSDE and FBSDE}).

We emphasise that, from above discussion, we know that for any stopping time $\tau$, it follows that
\begin{equation}\label{shifted SDE and SDE 1}
X_{\Theta+\tau}^{\Theta,x,v}(\omega^{\prime},\cdot)=X_{\tau}^{(\Theta),x,\hat{v}^{\omega^{\prime}}(\cdot,B^{\Theta})}
\quad \mathbb{P}\text{-}a.s., ~ \mathbb{P}^{\prime}(d\omega^\prime)\text{-}a.s.
\end{equation}
Moreover, for $\mathbb{P}^{\prime}$-almost all $\omega^{\prime}\in\Omega^{\prime}$, $\mathbb{P}$-a.s.
\begin{equation}\label{shifted exit time and exit time 1}
\tau_{\Theta,x,v}(\omega^{\prime},\cdot)= (\tau_{x,\hat{v}^{\omega^{\prime}}(\cdot,B)})(\pi_{\Theta(\omega^{\prime})})+\Theta(\omega^{\prime})
=\tau_{(\Theta),x,\hat{v}^{\omega^{\prime}}(\cdot,B^{\Theta})}+\Theta(\omega^{\prime}).
\end{equation}
Indeed
\[
\tau_{\Theta,x,v}(\omega^{\prime},\cdot)=\inf\{t\ge\Theta: X_{t}^{\Theta,x,v}\notin \overline{D}\}(\omega^{\prime},\cdot)
=\inf\{t\ge0: X_{\Theta+t}^{\Theta,x,v}(\omega^{\prime},\cdot)\notin \overline{D}\}+\Theta(\omega^{\prime}),
\]
and using (\ref{shifted exit time and exit time}) and (\ref{shifted SDE and SDE}), we obtain  $\mathbb{P}^{\prime}(d\omega^\prime)$-a.s., $\mathbb{P}$-a.s.,
\[
\begin{array}[c]{rl}
\tau_{\Theta,x,v}(\omega^{\prime},\cdot)
&=\inf\{t\ge0: X_{t}^{0,x,\hat{v}^{\omega^{\prime}}(\cdot,B)}\notin \overline{D}\}(\pi_{\Theta(\omega^{\prime})})+\Theta(\omega^{\prime})\medskip\\
&=(\tau_{x,\hat{v}^{\omega^{\prime}}(\cdot,B)})(\pi_{\Theta(\omega^{\prime})})+\Theta(\omega^{\prime})\medskip\\
&=\tau_{(\Theta),x,\hat{v}^{\omega^{\prime}}(\cdot,B^{\Theta})}+\Theta(\omega^{\prime}).
\end{array}
\]

\textbf{Step 3:} In this step we prove that  $\mathbb{P}^{\prime}(d\omega^\prime)$-a.s., $\mathbb{P}$-a.s.,
\[%\label{shifted BSDE and BSDE}
Y_{\Theta+t}^{\Theta,x,v}(\omega^{\prime},\cdot)=Y_{t}^{(\Theta),x,\hat{v}^{\omega^{\prime}}(\cdot,B^{\Theta})}
=(Y_{t}^{0,x,\hat{v}^{\omega^{\prime}}(\cdot,B)})(\pi_{\Theta(\omega^{\prime})}), \quad t\ge0.
\]
Using(\ref{shifted SDE and SDE 1}) and (\ref{shifted exit time and exit time 1}), the equation
\[
Y_{\Theta+t}^{\Theta,x,v}=\int_{(\Theta+t)\wedge\tau_{\Theta,x,v}}^{\tau_{\Theta,x,v}}
f(X_{s}^{\Theta,x,v},Y_{s}^{\Theta,x,v},Z_{s}^{\Theta,x,v},v_{s})ds-
\int_{(\Theta+t)\wedge\tau_{\Theta,x,v}}^{\tau_{\Theta,x,v}}Z_{s}^{\Theta,x,v}dB_{s}+g(X_{\tau_{\Theta,x,v}}^{\Theta,x,v}),
\]
$t\ge0$, takes the form
\[
\begin{array}[c]{rl}
Y_{\Theta+t}^{\Theta,x,v}(\omega^{\prime},\cdot)
&=\displaystyle\int_{t\wedge\tau_{(\Theta),x,\hat{v}^{\omega^{\prime}}(\cdot,B^{\Theta})}}^{\tau_{(\Theta),x,\hat{v}^{\omega^{\prime}}(\cdot,B^{\Theta})}}
f(X^{(\Theta),x,\hat{v}^{\omega^{\prime}}(\cdot,B^{\Theta})}_{s},
Y_{\Theta+s}^{\Theta,x,v}(\omega^{\prime},\cdot),Z_{\Theta+s}^{\Theta,x,v}(\omega^{\prime},\cdot),\hat{v}^{\omega^{\prime}}(s,B^{\Theta}))ds\medskip\\
&\qquad
-\displaystyle\int_{t\wedge\tau_{(\Theta),x,\hat{v}^{\omega^{\prime}}(\cdot,B^{\Theta})}}^{\tau_{(\Theta),x,\hat{v}^{\omega^{\prime}}(\cdot,B^{\Theta})}}
Z_{\Theta+s}^{\Theta,x,v}(\omega^{\prime},\cdot)dB_{s}^{\Theta}
+g(X^{(\Theta),x,\hat{v}^{\omega^{\prime}}(\cdot,B^{\Theta})}_{\tau_{(\Theta),x,\hat{v}^{\omega^{\prime}}(\cdot,B^{\Theta})}}), \quad t\ge0,
\end{array}
\]
$\mathbb{P}$-a.s., $\mathbb{P}^{\prime}(d\omega^\prime)$\text{-}a.s.,  and the uniqueness of the solution yields that
\[
Y_{\Theta+t}^{\Theta,x,v}(\omega^{\prime},\cdot)=Y_{t}^{(\Theta),x,\hat{v}^{\omega^{\prime}}(\cdot,B^{\Theta})},
\quad \mathbb{P}\text{-}a.s., ~ \mathbb{P}^{\prime}(d\omega^\prime)\text{-}a.s.
\]
Finally, (\ref{shifted FBSDE and FBSDE}) allows to conclude. Remark that in particular, for $t=0$, we have
\begin{equation}\label{shifted BSDE and BSDE time zero value}
Y_{\Theta}^{\Theta,x,v}(\omega^{\prime})=Y_{0}^{(\Theta),x,\hat{v}^{\omega^{\prime}}(\cdot,B^{\Theta})}
=(Y_{0}^{0,x,\hat{v}^{\omega^{\prime}}(\cdot,B)})(\pi_{\Theta(\omega^{\prime})}),
\quad \mathbb{P}^{\prime}(d\omega^\prime)\text{-}a.s.
\end{equation}
(Recall that $Y_{\Theta}^{\Theta,x,v}$ is $\mathcal{F}_{\Theta}$-measurable).

\textbf{Step 4:} Finally, we have
\[
u(x)=\inf_{\bar{v}\in\mathcal{V}}Y_{0}^{0,x,\bar{v}}=\mbox{essinf}_{v\in\mathcal{V}}Y_{\Theta}^{\Theta,x,v}, \quad\mathbb{P}\text{-}a.s.
\]
Indeed, let $v\in\mathcal{V}$. Then due to (\ref{shifted BSDE and BSDE time zero value}), $Y_{\Theta}^{\Theta,x,v}(\omega^{\prime})=Y_{0}^{0,x,\hat{v}^{\omega^{\prime}}(\cdot,B)}(\pi_{\Theta(\omega^{\prime})})$,  $\mathbb{P}^{\prime}(d\omega^\prime)\text{-}a.s.$ Recalling that $\hat{v}^{\omega^{\prime}}(\cdot,B)\in \mathcal{V}$, $\omega^\prime\in\Omega^\prime$,
and $u(x)$ as well as $Y_{0}^{0,x,\bar{v}}$, $\bar{v}\in\mathcal{V}$, are deterministic, it follows that
\[
Y_{\Theta}^{\Theta,x,v}(\omega)=Y_{\Theta}^{\Theta,x,v}(\omega^{\prime})\ge
\mbox{essinf}_{\bar{v}\in\mathcal{V}}(Y_{0}^{0,x,\bar{v}})(\pi_{\Theta(\omega^{\prime})})=u(x), \quad
\mathbb{P}(d\omega)\text{-}a.s.
\]
i.e., for the essential infimum under the probability $\mathbb{P}$,
\begin{equation}\label{equ 1}
\mbox{essinf}_{v\in\mathcal{V}}Y_{\Theta}^{\Theta,x,v}\ge u(x),\quad \mathbb{P}\text{-}a.s.
\end{equation}

On the other hand, let $\varepsilon>0$ and $v\in\mathcal{V}$ be such that $Y_{0}^{0,x,v}\leq u(x)+\varepsilon$ (Recall that $Y_{0}^{0,x,v}$ is deterministic). Then, for $\widetilde{v}$ which is a measurable, non-anticipating function on $\mathbb{R}_{+}\times\Omega$ such that $v=\widetilde{v}(\cdot,B)$, $drd\mathbb{P}$-a.e., using (\ref{shifted BSDE and BSDE time zero value}) we have
\[
u(x)+\varepsilon\ge Y_{0}^{0,x,v}=(Y_{0}^{0,x,\widetilde{v}(\cdot,B)})(\pi_{\Theta(\omega^{\prime})})
=Y_{\Theta}^{\Theta,x,\bar{v}}(\omega^\prime), \quad \mathbb{P}^{\prime}(d\omega^\prime)\text{-}a.s.
\]
and hence, for $\omega=(\omega^{\prime},\omega^{\prime\prime})$, $\mathbb{P}(d\omega)\text{-}a.s.$ Here, $\bar{v}\in\mathcal{V}$ is defined as follows: for some arbitrarily fixed $v_0\in V$, for $\omega=(\omega^{\prime},\omega^{\prime\prime})$,
\[
\bar{v}(s,\omega)=\bar{v}(s,\omega^{\prime},\omega^{\prime\prime})=\left\{
\begin{array}
[c]{rl}
& v_0,  \quad s\in[0,\Theta(\omega^\prime)],\medskip\\
& \widetilde{v}(s-\Theta(\omega^{\prime}),\omega^{\prime\prime})=\widetilde{v}(s-\Theta(\omega^{\prime}),B^{\Theta(\omega^\prime)}(\omega)), \quad s\in[\Theta(\omega^\prime),\infty).
\end{array}
\right.
\]
Consequently, with respect to the essinf under $\mathbb{P}$,
\[
u(x)+\varepsilon\ge \mbox{essinf}_{v\in\mathcal{V}}Y_{\Theta}^{\Theta,x,v},\quad \mathbb{P}\text{-}a.s.,
\]
and taking into account the arbitrariness of $\varepsilon>0$, we obtain
\[
u(x)\ge \mbox{essinf}_{v\in\mathcal{V}}Y_{\Theta}^{\Theta,x,v},\quad \mathbb{P}\text{-}a.s.
\]
Combined with (\ref{equ 1}), this yields the relation we had to show.
\end{proof}
\begin{lemma}\label{lemma 1 for proving DPP}
Under the assumptions of Theorem \ref{DPP},  let $\Theta$ be a stopping time with $Ee^{\mu \Theta}<\infty$ and $\xi\in L^{2}(\mathcal{F}_{\Theta};\mathbb{R}^{d})$. Then, for all $v\in\mathcal{V}$, we have
\begin{equation}\label{inequality 1 for proving DPP}
u(\xi)\leq Y_{\Theta}^{\Theta,\xi,v}, \quad \mathbb{P}\text{-}a.s.
\end{equation}
Conversely, for all $\varepsilon>0$, there exists $v^{\varepsilon}\in\mathcal{V}$, such that
\begin{equation}\label{inequality 2 for proving DPP}
u(\xi)+\varepsilon\ge Y_{\Theta}^{\Theta,\xi,v^{\varepsilon}},\quad \mathbb{P}\text{-}a.s.
\end{equation}
\end{lemma}
\begin{proof}
Let $\xi,\xi^{\prime}\in L^{2}(\mathcal{F}_{\Theta};\mathbb{R}^{d})$. Then with the notations introduced in the proof of Theorem \ref{Lemma for equivalent form of value function},  we have $\xi(\omega)=\xi(\omega^{\prime})$ and $\xi^{\prime}(\omega)=\xi^{\prime}(\omega^{\prime})$, for $\omega\equiv(\omega^{\prime},\omega^{\prime\prime})\in\Omega^{\prime}\otimes\Omega^{\prime\prime}$.
Therefore, for $Y^{\Theta,\xi,v}$ and $Y^{\Theta,\xi^\prime,v}$ defined in Lemma \ref{Lemma for equivalent form of value function}, similarly to (\ref{shifted BSDE and BSDE time zero value}) we see that
\[
Y_{\Theta}^{\Theta,\xi,v}(\omega^{\prime})=Y_{0}^{(\Theta),\xi(\omega^{\prime}),\hat{v}^{\omega^{\prime}}(\cdot,B^{\Theta})}
=Y_{0}^{0,\xi(\omega^{\prime}),\hat{v}^{\omega^{\prime}}(\cdot,B)}(\pi_{\Theta(\omega^{\prime})}),
\quad \mathbb{P}^{\prime}(d\omega^\prime)\text{-}a.s.,
\]
and
\[
Y_{\Theta}^{\Theta,\xi^\prime,v}(\omega^{\prime})=Y_{0}^{(\Theta),\xi^\prime(\omega^{\prime}),\hat{v}^{\omega^{\prime}}(\cdot,B^{\Theta})}
=Y_{0}^{0,\xi^\prime(\omega^{\prime}),\hat{v}^{\omega^{\prime}}(\cdot,B}(\pi_{\Theta(\omega^{\prime})}),
\quad \mathbb{P}^{\prime}(d\omega^\prime)\text{-}a.s.
\]
Then, for all $v\in\mathcal{V}$, $\mathbb{P}^{\prime}(d\omega^{\prime})$-a.s.,
\[
\begin{array}[c]{rl}
\left|Y_{\Theta}^{\Theta,\xi,v}(\omega^{\prime})-Y_{\Theta}^{\Theta,\xi^{\prime},v}(\omega^{\prime})\right|
&=\left|\left(Y_{0}^{0,\xi(\omega^{\prime}),\hat{v}^{\omega^{\prime}}(\cdot,B)}
-Y_{0}^{0,\xi^{\prime}(\omega^{\prime}),\hat{v}^{\omega^{\prime}}(\cdot,B)}\right)(\pi_{\Theta(\omega^{\prime})})\right|\medskip\\
&=\left|Y_{0}^{0,\xi(\omega^{\prime}),\hat{v}^{\omega^{\prime}}(\cdot,B)}
-Y_{0}^{0,\xi^{\prime}(\omega^{\prime}),\hat{v}^{\omega^{\prime}}(\cdot,B)}\right|
\end{array}
\]
where we used the fact that for fixed $\omega^{\prime}\in\Omega^{\prime}$, $Y_{0}^{0,\xi(\omega^{\prime}),\hat{v}^{\omega^{\prime}}(\cdot,B)}$ and $Y_{0}^{0,\xi^{\prime}(\omega^{\prime}),\hat{v}^{\omega^{\prime}}(\cdot,B)}$ are deterministic. On the other hand, from Theorem \ref{regularity of u} it follows
\[
\left|Y_{0}^{0,\xi(\omega^{\prime}),\hat{v}^{\omega^{\prime}}(\cdot,B)}
-Y_{0}^{0,\xi^{\prime}(\omega^{\prime}),\hat{v}^{\omega^{\prime}}(\cdot,B)}\right|
\leq C|\xi(\omega^{\prime})-\xi^{\prime}(\omega^{\prime})|^{1/2},
\]
for a constant $C$ independent of $\omega^\prime\in\Omega^\prime$. Consequently,  for all $\xi,\xi^{\prime}\in L^{2}(\mathcal{F}_{\Theta};\mathbb{R}^{d})$ and $v\in\mathcal{V}$,
\[
\left|Y_{\Theta}^{\Theta,\xi,v}-Y_{\Theta}^{\Theta,\xi^{\prime},v}\right|\leq
C|\xi-\xi^{\prime}|^{1/2},\quad \mathbb{P}\text{-}a.s.
\]
Thus, in order to prove (\ref{inequality 1 for proving DPP}), we only need to show that $u(\xi)\leq Y_{\Theta}^{\Theta,\xi,v}$, $\mathbb{P}$-a.s., for all $\xi$ taking the form
$\xi=\sum_{i=1}^{\infty}1_{A_{i}}x_{i}$,
where $\{A_{i}\}_{i=1}^{\infty}$ is a partition of $(\Omega,\mathcal{F}_{\Theta})$ and $x_{i}\in\mathbb{R}^{d}$, $i\ge1$. Following the argument of Peng \cite{P-97}, the uniqueness of the solution of SDE and BSDE with initial data $(\Theta,\xi)$ yields
\[
Y_{\Theta}^{\Theta,\xi,v}=\sum_{i=1}^{\infty}1_{A_{i}}Y_{\Theta}^{\Theta,x_{i},v}.
\]
From Lemma \ref{Lemma for equivalent form of value function}
we know $u(x)=\mbox{essinf}_{v\in\mathcal{V}}Y_{\Theta}^{\Theta,x,v}$, $x\in\mathbb{R}^{d}$.
Hence,
\[
u(\xi)=u(\sum_{i=1}^{\infty}1_{A_{i}}x_{i})=\sum_{i=1}^{\infty}1_{A_{i}}u(x_{i})
\leq \sum_{i=1}^{\infty}1_{A_{i}} Y_{\Theta}^{\Theta,x_{i},v}=Y_{\Theta}^{\Theta,\xi,v},\quad \mathbb{P}\text{-}a.s.
\]
We have proved (\ref{inequality 1 for proving DPP}). Now let us show (\ref{inequality 2 for proving DPP}). For $\xi\in L^{2}(\mathcal{F}_{\Theta};\mathbb{R}^{d})$ we construct the random variable
$\eta:=\sum_{i=1}^{\infty}1_{A_{i}}x_{i}\in L^{2}(\mathcal{F}_{\Theta};\mathbb{R}^{d})$, where $\{A_{i}\}_{i=1}^{\infty}$ is a partition of $(\Omega,\mathcal{F}_{\Theta})$ and $x_{i}\in\mathbb{R}^{d}$, $i\ge1$  s.t.
\[
|\eta-\xi|\leq \frac{1}{C^{2}}\left(\frac{\varepsilon}{3}\right)^{2},
\]
where $C$ is the constant as in Theorem \ref{regularity of u}.
Then, from the 1/2-H\"{o}lder continuity of $u(x)$ and $Y_{\Theta}^{\Theta,x,v}$ w.r.t. $x$ we have
\begin{equation}\label{estimate 1}
|u(\xi)-u(\eta)|\leq \frac{\varepsilon}{3}, \quad |Y_{\Theta}^{\Theta,\xi,v}-Y_{\Theta}^{\Theta,\eta,v}|\leq\frac{\varepsilon}{3}, \quad a.s.
\end{equation}
From Lemma \ref{Lemma for equivalent form of value function} we know $u(x)=\mbox{essinf}_{v\in\mathcal{V}}Y_{\Theta}^{\Theta,x,v}$, $\mathbb{P}$-a.s.
Thus, for every $i\ge1$, there exist a sequence $\{v^{i,j}\}_{j\ge1}\subset\mathcal{V}$ such that
$u(x_{i})=\inf_{j\ge1}Y_{\Theta}^{\Theta,x_{i},v^{i,j}}$, $\mathbb{P}$-a.s.
We define $\tilde{\Gamma}_{i,j}:=\{u(x_{i})+\frac{\varepsilon}{3}\ge Y_{\Theta}^{\Theta,x_{i},v^{i,j}}\}\in \mathcal{F}_{\Theta}$, $j\ge1$. Then
$\Gamma_{i,1}:=\tilde{\Gamma}_{i,1}$, $\Gamma_{i,j}:=\tilde{\Gamma}_{i,j}\setminus\cup_{l=1}^{j-1}\tilde{\Gamma}_{i,l}$, $j\ge2$, is a  partition of $(\Omega,\mathcal{F}_{\Theta})$.  Let $v^{i,\varepsilon}:=\sum_{j\ge1}1_{\Gamma_{i,j}}v^{i,j}\in \mathcal{V}$. Then, following  again Peng's argument \cite{P-97}, we have $Y_{\Theta}^{\Theta,x_{i},v^{i,\varepsilon}}=\sum_{j\ge1}1_{\Gamma_{i,j}}Y_{\Theta}^{\Theta,x_{i},v^{i,j}}$. Thus
\[
Y_{\Theta}^{\Theta,x_{i},v^{i,\varepsilon}}=\sum_{j\ge1}1_{\Gamma_{i,j}}Y_{\Theta}^{\Theta,x_{i},v^{i,j}}
\leq \sum_{j\ge1}1_{\Gamma_{i,j}}u(x_{i})+\frac{\varepsilon}{3}=u(x_{i})+\frac{\varepsilon}{3}, \quad\mathbb{P}\text{-}a.s.
\]
Consequently, if we put
$v^{\varepsilon}:=\sum_{i\ge1}1_{A_{i}}v^{i,\varepsilon}=\sum_{i\ge1}\sum_{j\ge1}1_{A_{i}\cap\Gamma_{i,j}}v^{i,j}\in \mathcal{V}$,
then we have $\sum_{i=1}^{\infty}1_{A_{i}}Y_{\Theta}^{\Theta,x_{i},v^{i,\varepsilon}}=Y_{\Theta}^{\Theta,\eta,v^{\varepsilon}}$ (see, e.g. \cite{P-97}), and from above inequality combined with (\ref{estimate 1}) it follows
\[
\begin{array}{rl}
u(\xi)&\ge u(\eta)-\frac{\varepsilon}{3}=u(\sum_{i=1}^{\infty}1_{A_{i}}x_{i})-\frac{\varepsilon}{3}
=\sum_{i=1}^{\infty}1_{A_{i}}u(x_{i})-\frac{\varepsilon}{3}\medskip\\
&\ge \sum_{i=1}^{\infty}1_{A_{i}}(Y_{\Theta}^{\Theta,x_{i},v^{i,\varepsilon}}-\frac{\varepsilon}{3})-\frac{\varepsilon}{3}
=\sum_{i=1}^{\infty}1_{A_{i}}Y_{\Theta}^{\Theta,x_{i},v^{i,\varepsilon}}-\frac{2\varepsilon}{3}\medskip\\
&=Y_{\Theta}^{\Theta,\eta,v^{\varepsilon}}-\frac{2\varepsilon}{3}
\ge Y_{\Theta}^{\Theta,\xi,v^{\varepsilon}}-\frac{\varepsilon}{3}-\frac{2\varepsilon}{3}
= Y_{\Theta}^{\Theta,\xi,v^{\varepsilon}}-\varepsilon, \quad \mathbb{P}\text{-}a.s.
\end{array}
\]
Therefore, we have found a $v^{\varepsilon}\in\mathcal{V}$, such that (\ref{inequality 2 for proving DPP}) holds.
\end{proof}
\smallskip

Under the assumption of Theorem \ref{DPP} we have the following both lemmas concerning the sub- and super-dynamic programming principle.
\begin{lemma}\label{lemma 2 for proving DPP}
Let $\Theta$ be a stopping time with $Ee^{\mu \Theta}<\infty$. Then,
$
u(x)\ge\inf\limits_{v\in\mathcal{V}}G_{\tau_{x,v}\wedge\Theta}^{0,x,v}[u(X_{\tau_{x,v}\wedge\Theta}^{0,x,v})].
$
\end{lemma}
\begin{proof}
Recalling the definition of our value function and that of the backward semigroup (see (\ref{value function}) and (\ref{backward semigroup})), we obtain
\[
u(x)=\inf\limits_{v\in\mathcal{V}}Y_{0}^{0,x,v}=\inf\limits_{v\in\mathcal{V}}G_{\tau_{x,v}}^{0,x,v}[g(X_{\tau_{x,v}}^{0,x,v})]
=\inf\limits_{v\in\mathcal{V}}G_{\tau_{x,v}\wedge\Theta}^{0,x,v}[Y_{\tau_{x,v}\wedge\Theta}^{0,x,v}].
\]
From the uniqueness of the solution of the SDE and the BSDE with initial data $(\tau_{x,v}\wedge\Theta,X_{\tau_{x,v}\wedge\Theta}^{0,x,v})$ combined with Lemma \ref{lemma 1 for proving DPP} (\ref{inequality 1 for proving DPP}) we get
\[
Y_{\tau_{x,v}\wedge\Theta}^{0,x,v}=Y_{\tau_{x,v}\wedge\Theta}^{\tau_{x,v}\wedge\Theta,X_{\tau_{x,v}\wedge\Theta}^{0,x,v},v}
\ge u(X_{\tau_{x,v}\wedge\Theta}^{0,x,v}),\quad\mathbb{P}\text{-}a.s.
\]
Finally, the comparison theorem for BSDEs (see Lemma \ref{comparison theorem}) yields
\[
u(x)
=\inf\limits_{v\in\mathcal{V}}G_{\tau_{x,v}\wedge\Theta}^{0,x,v}
[Y_{\tau_{x,v}\wedge\Theta}^{\tau_{x,v}\wedge\Theta,X_{\tau_{x,v}\wedge\Theta}^{0,x,v},v}]
\ge
\inf\limits_{v\in\mathcal{V}}G_{\tau_{x,v}\wedge\Theta}^{0,x,v}[u(X_{\tau_{x,v}\wedge\Theta}^{0,x,v})].
\]
\hfill
\end{proof}
\begin{lemma}\label{lemma 3 for proving DPP}
Under the same assumption as in Lemma \ref{lemma 2 for proving DPP} we have
\[
u(x)\leq\inf\limits_{v\in\mathcal{V}}G_{\tau_{x,v}\wedge\Theta}^{0,x,v}[u(X_{\tau_{x,v}\wedge\Theta}^{0,x,v})].
\]
\end{lemma}
\begin{proof}
From Lemma \ref{lemma 1 for proving DPP} (\ref{inequality 2 for proving DPP}), we know that, for arbitrary $\varepsilon>0$, there exists $v^{\varepsilon}\in\mathcal{V}$ such that
\[
u(X_{\tau_{x,v}\wedge\Theta}^{0,x,v})\ge Y_{\tau_{x,v}\wedge\Theta}^{\tau_{x,v}\wedge\Theta,X_{\tau_{x,v}\wedge\Theta}^{0,x,v},v^{\varepsilon}}-\varepsilon.
\]
Then from the comparison theorem for BSDEs it follows that
\[
\inf\limits_{v\in\mathcal{V}}G_{\tau_{x,v}\wedge\Theta}^{0,x,v}[u(X_{\tau_{x,v}\wedge\Theta}^{0,x,v})]
\ge \inf\limits_{v\in\mathcal{V}}G_{\tau_{x,v}\wedge\Theta}^{0,x,v}
[Y_{\tau_{x,v}\wedge\Theta}^{\tau_{x,v}\wedge\Theta,X_{\tau_{x,v}\wedge\Theta}^{0,x,v},v^{\varepsilon}}-\varepsilon].
\]
With the help of Lemma \ref{Stability w.r.t. perturbation} and the definition of backward semigroup we deduce that there exists a constant $C$ independent of $\varepsilon$ s.t.
\[
\inf\limits_{v\in\mathcal{V}}G_{\tau_{x,v}\wedge\Theta}^{0,x,v}
[Y_{\tau_{x,v}\wedge\Theta}^{\tau_{x,v}\wedge\Theta,X_{\tau_{x,v}\wedge\Theta}^{0,x,v},v^{\varepsilon}}-\varepsilon]
\ge \inf\limits_{v\in\mathcal{V}}G_{\tau_{x,v}\wedge\Theta}^{0,x,v}
[Y_{\tau_{x,v}\wedge\Theta}^{\tau_{x,v}\wedge\Theta,X_{\tau_{x,v}\wedge\Theta}^{0,x,v},v^{\varepsilon}}]-C\varepsilon.
\]
One the other hand, as already indicated in the proof of Lemma \ref{lemma 2 for proving DPP},
\[
u(x)
=\inf\limits_{v\in\mathcal{V}}G_{\tau_{x,v}\wedge\Theta}^{0,x,v}[Y_{\tau_{x,v}\wedge\Theta}^{0,x,v}]
=\inf\limits_{v\in\mathcal{V}}G_{\tau_{x,v}\wedge\Theta}^{0,x,v}
[Y_{\tau_{x,v}\wedge\Theta}^{\tau_{x,v}\wedge\Theta,X_{\tau_{x,v}\wedge\Theta}^{0,x,v},v}],
\]
so we have that, by combining the above estimates,
\[
\inf\limits_{v\in\mathcal{V}}G_{\tau_{x,v}\wedge\Theta}^{0,x,v}[u(X_{\tau_{x,v}\wedge\Theta}^{0,x,v})]
\ge \inf\limits_{v\in\mathcal{V}}G_{\tau_{x,v}\wedge\Theta}^{0,x,v}
[Y_{\tau_{x,v}\wedge\Theta}^{\tau_{x,v}\wedge\Theta,X_{\tau_{x,v}\wedge\Theta}^{0,x,v},v^{\varepsilon}}]-C\varepsilon
\ge u(x)-C\varepsilon.
\]
Finally, since $\varepsilon$ is arbitrary, the proof is completed.
\end{proof}
\medskip

Remark that the Lemmas \ref{lemma 2 for proving DPP} and \ref{lemma 3 for proving DPP} just prove Theorem \ref{DPP}.

\section{Generalized HJB equation with Dirichlet boundary}
In this section we consider the following generalized Hamilton-Jacobi-Bellman equation with Dirichlet boundary:
\begin{equation}\label{HJB equation}
\left\{
\begin{array}
[c]{l}
\inf\limits_{v\in V}\left\{\mathcal{L}(x,v)u(x)+f(x,u(x),\nabla u(x) \sigma(x,v),v)\right\}=0, \quad x\in D,\medskip\\
u(x)=g(x),\quad x\in \partial D,
\end{array}
\right.
\end{equation}
where $D$ is the bounded domain in $\mathbb{R}^{d}$, and $V$ is the compact metric space in $\mathbb{R}^{k}$, introduced in Section 2. For $u\in C^{2}(D)$ and $(x,v)\in D\times V$,  we have put
\[\mathcal{L}(x,v)u(x):=\frac{1}{2}\sum_{i,j=1}^{d}(\sigma\sigma^{\ast})_{i,j}(x,v)\frac{\partial^{2}u}{\partial x_{i}\partial x_{j}}(x)
+\sum_{i=1}^{d}b_{i}(x,v)\frac{\partial u}{\partial x_{i}}(x),
\]
and we suppose that the coefficients $b,\sigma$ and $f$ satisfy the assumptions $(H_{1})$-$(H_{5})$ and that $g\in C(\overline{D})$.

First, let us recall the definition of a viscosity solution of (\ref{HJB equation}); see Crandall, Ishii and Lions \cite{CIL-92} for more details.
\begin{definition}\label{Definition of viscosity solution}
(i) A continuous function $u:\overline{D}\rightarrow\mathbb{R}$ is called a viscosity subsolution of (\ref{HJB equation}),  if $u(x)\leq g(x)$, for all $x\in\partial D$, and if, for any $\varphi\in C^{2}(\overline{D})$ and any local maximum point $x$ of $u-\varphi$, it holds that
\[
\begin{array}{ll}
\inf\limits_{v\in V}\left\{\mathcal{L}(x,v)\varphi(x)+f(x,u(x),\nabla \varphi(x) \sigma(x,v),v)\right\}\ge 0, \quad
  x\in \overline{D}\setminus\partial D.
\end{array}
\]
(ii) The function $u$ is called a viscosity supersolution of (\ref{HJB equation}),  if $u(x)\ge g(x)$, for all $x\in\partial D$, and if, for any $\varphi\in C^{2}(\overline{D})$ and any local minimum point $x$ of $u-\varphi$, we have
\[
\begin{array}{ll}
\inf\limits_{v\in V}\left\{\mathcal{L}(x,v)\varphi(x)+f(x,u(x),\nabla \varphi(x) \sigma(x,v),v)\right\}\leq 0, \quad
 x\in \overline{D}\setminus\partial D.
\end{array}
\]
(iii) The function $u$ is said to be a viscosity solution of (\ref{HJB equation}), if it is both a viscosity subsolution and a viscosity supersolution of (\ref{HJB equation}).
\end{definition}
\begin{remark}
Standard arguments show that it is sufficient to consider test functions in Definition \ref{Definition of viscosity solution} which belong to $C^{3}(\overline{D})$, see for instance \cite{L-1983} Remark I.9 or \cite{G-2006} Proposition 2.2.3.
\end{remark}
In this section we assume that
\begin{itemize}
\item[$\left(H_{6}\right)$] $f(x,y,z,v)$ is Lipschitz continuous  w.r.t. $y$, uniformly on $(x,z,v)$, i.e.  there exists a constants $\tilde{L}\ge0$, such that, for all $(x,z,v)\in \overline{D}\times\mathbb{R}^{m}\times V$, $y_{1},y_{2}\in\mathbb{R}$,
    \[
      |f(x,y_{1},z,v)-f(x,y_{2},z,v)|\leq \tilde{L}|y_{1}-y_{2}|.
    \]
\end{itemize}
We would like to show that the value function $u(x)$ (see (\ref{value function})) of our stochastic exit time optimal control problem introduced in Section 2 is the viscosity solution of (\ref{HJB equation}). Motivated by the BSDE approach of Peng \cite{P-97}, we first give several auxiliary lemmas. First, for arbitrary but fixed $\varphi\in C^{3}(\overline{D})$, we set
\[
F(x,y,z,v):=\mathcal{L}(x,v)\varphi(x)+f(x,y+\varphi(x),z+\nabla\varphi(x)\sigma(x,v),v),
\]
$(x,y,z,v)\in\mathbb{R}^{d}\times\mathbb{R}\times\mathbb{R}^{m}\times V$. Recalling that $X^{0,x,v}$ is the solution of SDE (\ref{SDE}) and the stochastic exit time $\tau_{x,v}$ is defined in (\ref{exit time}), we consider the following BSDE with random terminal time $\tau_{x,v}\wedge \varepsilon$,  for an arbitrary but fixed $0<\varepsilon\leq 1$:
\begin{equation}\label{BSDE Y1}
\left\{
\begin{array}
[c]{l}
-dY_{s}^{1;0,x,v;\varepsilon}=F(X_{s}^{0,x,v},Y_{s}^{1;0,x,v;\varepsilon}, Z_{s}^{1;0,x,v;\varepsilon},v_{s})ds-Z_{s}^{1;0,x,v;\varepsilon}dB_{s},
\quad 0\leq s\leq\tau_{x,v}\wedge\varepsilon,\medskip\\
Y_{\tau_{x,v}\wedge \varepsilon}^{1;0,x,v;\varepsilon}=0.
\end{array}
\right.
\end{equation}

\begin{lemma}\label{lemma 1 for proving viscosity solution}
Under the assumptions $(H_{1})$-$(H_{6})$, BSDE (\ref{BSDE Y1}) has a unique solution $(Y^{1;0,x,v;\varepsilon},$
$Z^{1;0,x,v;\varepsilon})\in M_{\gamma}^{2}(0,\tau_{x,v}\wedge\varepsilon;\mathbb{R})\times M_{\gamma}^{2}(0,\tau_{x,v}\wedge\varepsilon;\mathbb{R}^{m})$.
The solution also belongs to
$M_{\mu}^{2}(0,\tau_{x,v}\wedge\varepsilon;\mathbb{R})\times M_{\mu}^{2}(0,\tau_{x,v}\wedge\varepsilon;\mathbb{R}^{m})$ and satisfies
$E[\sup\limits_{0\leq s\leq \tau_{x,v}\wedge\varepsilon}e^{\mu s}|Y^{1;0,x,v;\varepsilon}_{s}|^{2}]<\infty$. Moreover, we have
\begin{equation}\label{estimate 2}
Y_{s\wedge\tau_{x,v}\wedge\varepsilon}^{1;0,x,v;\varepsilon}=G_{s,\tau_{x,v}\wedge\varepsilon}^{0,x,v}
[\varphi(X_{\tau_{x,v}\wedge\varepsilon}^{0,x,v})]
-\varphi(X_{s\wedge\tau_{x,v}\wedge\varepsilon}^{0,x,v}),\quad s\ge0,~\mathbb{P}\text{-}a.s.
\end{equation}
\end{lemma}
\begin{proof}
It is direct to verify that $F(X_{s}^{0,x,v},y,z,v)$ and $\tau_{x,v}\wedge\varepsilon$ satisfy the conditions of Lemma \ref{Wellposedness Darling and Pardoux}. So we know that BSDE (\ref{BSDE Y1}) has a unique solution $(Y^{1;0,x,v;\varepsilon},Z^{1;0,x,v;\varepsilon})\in M_{\gamma}^{2}(0,\tau_{x,v}\wedge\varepsilon;\mathbb{R})\times M_{\gamma}^{2}(0,\tau_{x,v}\wedge\varepsilon;\mathbb{R}^{m})$. Moreover, the solution belongs to
$M_{\mu}^{2}(0,\tau_{x,v}\wedge\varepsilon;\mathbb{R})\times M_{\mu}^{2}(0,\tau_{x,v}\wedge\varepsilon;\mathbb{R}^{m})$ and satisfies
$E[\sup\limits_{0\leq s\leq \tau_{x,v}\wedge\varepsilon}e^{\mu s}|Y^{1;0,x,v;\varepsilon}_{s}|^{2}]<\infty$.

It remains to show (\ref{estimate 2}). We recall that
$G_{s,\tau_{x,v}\wedge\varepsilon}^{0,x,v}[\varphi(X_{\tau_{x,v}\wedge\varepsilon}^{0,x,v})]
:=Y_{s\wedge\tau_{x,v}}^{\varphi;0,x,v;\varepsilon}$,
where $(Y^{\varphi;0,x,v;\varepsilon}$,
$Z^{\varphi;0,x,v;\varepsilon})$ is the solution of the following BSDE
\[
\left\{
\begin{array}
[c]{l}
-dY_{s}^{\varphi;0,x,v;\varepsilon}=f(X_{s}^{0,x,v},Y_{s}^{\varphi;0,x,v;\varepsilon}, Z_{s}^{\varphi;0,x,v;\varepsilon},v_{s})ds-Z_{s}^{\varphi;0,x,v;\varepsilon}dB_{s}, \quad 0\leq s\leq\tau_{x,v}\wedge\varepsilon, \medskip\\
Y_{\tau_{x,v}\wedge\varepsilon}^{\varphi;0,x,v;\varepsilon}=\varphi(X_{\tau_{x,v}\wedge\varepsilon}^{0,x,v}).
\end{array}
\right.
\]
Therefore, we only need to show that
$Y_{s\wedge\tau_{x,v}\wedge\varepsilon}^{\varphi;0,x,v;\varepsilon}-\varphi(X_{s\wedge\tau_{x,v}\wedge\varepsilon}^{0,x,v})
=Y_{s\wedge\tau_{x,v}\wedge\varepsilon}^{1;0,x,v;\varepsilon;\varepsilon}$.
But this relation holds true, it can be verified easily by applying It\^{o}'s formula to $\varphi(X_{s}^{0,x,v})$ and by considering that at terminal time $\tau_{x,v}\wedge\varepsilon$,
$Y_{\tau_{x,v}\wedge\varepsilon}^{\varphi;0,x,v;\varepsilon}-\varphi(X_{\tau_{x,v}\wedge\varepsilon}^{0,x,v})
=0=Y_{\tau_{x,v}\wedge\varepsilon}^{1;0,x,v;\varepsilon}$.
\end{proof}

\begin{lemma}\label{lemma 2 for proving viscosity solution}
For the solution $(Y^{2;0,x,v;\varepsilon},Z^{2;0,x,v;\varepsilon})$ of the following simple BSDE
\begin{equation}\label{BSDE Y2}
\left\{
\begin{array}
[c]{l}
-dY_{s}^{2;0,x,v;\varepsilon}=F(x,Y_{s}^{2;0,x,v;\varepsilon}, Z_{s}^{2;0,x,v;\varepsilon},v_{s})ds-Z_{s}^{2;0,x,v;\varepsilon}dB_{s},
\quad 0\leq s\leq\tau_{x,v}\wedge\varepsilon, \medskip\\
Y_{\tau_{x,v}\wedge\varepsilon}^{2;0,x,v;\varepsilon}=0,
\end{array}
\right.
\end{equation}
there exists a constant $C$ independent of $v$, $\varepsilon$ and $x\in\overline{D}$, such that
\begin{equation}\label{estimate between Y1 and Y2}
|Y_{0}^{1;0,x,v;\varepsilon}-Y_{0}^{2;0,x,v;\varepsilon}|\leq C \varepsilon^{\frac{3}{2}},
\end{equation}
and
\begin{equation}\label{estimate for Y2}
E\left[\int_{0}^{\tau_{x,v}\wedge\varepsilon}\left(|Y_{s}^{2;0,x,v;\varepsilon}|+|Z_{s}^{2;0,x,v;\varepsilon}|\right)ds\right]
\leq C\varepsilon^{\frac{3}{2}}.
\end{equation}
\end{lemma}
\begin{proof}
Let us first show (\ref{estimate between Y1 and Y2}). As $b$ and $\sigma$ are bounded over $\overline{D}\times V$, we have for all $\varepsilon>0$, $v\in\mathcal{V}$, $x\in\overline{D}$ and $p\ge2$,
\begin{equation}\label{estimate for SDE}
\begin{array}{rl}
E\left[\sup\limits_{t\in[0,\varepsilon]}|X_{t}^{0,x,v}-x|^{p}\right]\leq &2^{p-1}E\left[\sup\limits_{t\in[0,\varepsilon]}\left|\displaystyle\int_{0}^{t}b(X_{s}^{0,x,v},v_{s})ds\right|^{p}\right] \medskip\\
& +2^{p-1}E\left[\sup\limits_{t\in[0,\varepsilon]}\left|\displaystyle\int_{0}^{t}\sigma(X_{s}^{0,x,v},v_{s})dB_{s}\right|^{p}\right] \medskip\\
\leq & C_{p}\varepsilon^{p/2}.
\end{array}
\end{equation}
Now, we apply Lemma \ref{Stability w.r.t. perturbation} to the BSDEs (\ref{BSDE Y1}) and (\ref{BSDE Y2}). Then for all $\theta\in(\beta^{2}-2\alpha,\mu]$ and for some constant $C$ independent of $v$ and $\varepsilon$,
\[
\begin{array}{rl}
&E\left[\displaystyle\int_{0}^{\tau_{x,v}\wedge\varepsilon}e^{\theta s}\left(|Y_{s}^{1;0,x,v;\varepsilon}-Y_{s}^{2;0,x,v;\varepsilon}|^{2}
+|Z_{s}^{1;0,x,v;\varepsilon}-Z_{s}^{2;0,x,v;\varepsilon}|^{2}\right)ds\right]\medskip\\
\leq & CE\left[\displaystyle\int_{0}^{\tau_{x,v}\wedge\varepsilon}e^{\theta s}|F(X_{s}^{0,x,v},Y_{s}^{2;0,x,v;\varepsilon}, Z_{s}^{2;0,x,v;\varepsilon},v_{s})-F(x,Y_{s}^{2;0,x,v;\varepsilon}, Z_{s}^{2;0,x,v;\varepsilon},v_{s})|^{2}ds\right].
\end{array}
\]
As we know from Lemma \ref{lemma for assumption} that there is a positive $\mu>0$, we can take a positive $\theta$ in above inequality.  Moreover,
from the assumptions $(H_{1})$, $(H_{3})$ and $(H_{6})$, we have
\[
\begin{array}{rl}
&|F(X_{s}^{0,x,v},Y_{s}^{2;0,x,v;\varepsilon}, Z_{s}^{2;0,x,v;\varepsilon},v_{s})-F(x,Y_{s}^{2;0,x,v;\varepsilon}, Z_{s}^{2;0,x,v;\varepsilon},v_{s})|\medskip\\
\leq & C(1+|x|^{2})(|X_{s}^{0,x,v}-x|+|X_{s}^{0,x,v}-x|^{2})\medskip\\
\leq & C(|X_{s}^{0,x,v}-x|+|X_{s}^{0,x,v}-x|^{2}),\quad 0\leq s\leq\tau_{x,v}.
\end{array}
\]
(Recall that $\overline{D}$ is bounded). Therefore,
\[
\begin{array}{rl}
&E\left[\displaystyle\int_{0}^{\tau_{x,v}\wedge\varepsilon}\left(|Y_{s}^{1;0,x,v;\varepsilon}-Y_{s}^{2;0,x,v;\varepsilon}|^{2}
+|Z_{s}^{1;0,x,v;\varepsilon}-Z_{s}^{2;0,x,v;\varepsilon}|^{2}\right)ds\right]\medskip\\
\leq & C\varepsilon e^{\mu\varepsilon} E\left[\sup\limits_{t\in[0,\varepsilon]}(|X_{t}^{0,x,v}-x|^{2}+|X_{t}^{0,x,v}-x|^{4})\right]
\leq C\varepsilon e^{\mu\varepsilon}(\varepsilon+\varepsilon^{2}).
\end{array}
\]
Consequently, recalling that both $Y_{0}^{1;0,x,v;\varepsilon}$ and  $Y_{0}^{2;0,x,v;\varepsilon}$ are deterministic,  we have
\[
\begin{array}{rl}
&|Y_{0}^{1;0,x,v;\varepsilon}-Y_{0}^{2;0,x,v;\varepsilon}|=\left|E\left[Y_{0}^{1;0,x,v;\varepsilon}-Y_{0}^{2;0,x,v;\varepsilon}\right]\right|\medskip\\
&=\left|E\left[\displaystyle\int_{0}^{\tau_{x,v}\wedge\varepsilon}(F(X_{s}^{0,x,v},Y_{s}^{1;0,x,v;\varepsilon}, Z_{s}^{1;0,x,v;\varepsilon},v_{s})-F(x,Y_{s}^{2;0,x,v;\varepsilon}, Z_{s}^{2;0,x,v;\varepsilon},v_{s}))ds\right]\right|\medskip\\
&\leq CE\left[\displaystyle\int_{0}^{\tau_{x,v}\wedge\varepsilon}\left(|X_{s}^{0,x,v}-x|+|X_{s}^{0,x,v}-x|^{2}\right)ds\right]\medskip\\
&\qquad\qquad+CE\left[\displaystyle\int_{0}^{\tau_{x,v}\wedge\varepsilon}\left(|Y_{s}^{1;0,x,v;\varepsilon}- Y_{s}^{2;0,x,v;\varepsilon}|+|Z_{s}^{1;0,x,v;\varepsilon}-Z_{s}^{2;0,x,v;\varepsilon}|\right)ds\right]\medskip\\
&\leq C\varepsilon(\varepsilon^{\frac{1}{2}}+\varepsilon)
+C\varepsilon^{\frac{1}{2}}\left\{E\left[\displaystyle\int_{0}^{\tau_{x,v}\wedge\varepsilon}\left(|Y_{s}^{1;0,x,v;\varepsilon}- Y_{s}^{2;0,x,v;\varepsilon}|^{2}+|Z_{s}^{1;0,x,v;\varepsilon}-Z_{s}^{2;0,x,v;\varepsilon}|^{2}\right)ds\right]\right\}^{\frac{1}{2}}\medskip\\
&\leq C\varepsilon^{\frac{3}{2}}.
\end{array}
\]
Now we are going to prove (\ref{estimate for Y2}). For this end, we apply It\^{o}'s formula to $|Y_{s}^{2;0,x,v;\varepsilon}|^{2}$. Recalling that $F(x,\cdot,\cdot,v)$ has a linear growth in $(y,z)$, uniformly in $(x,v)\in \overline{D}\times V$, we obtain
\[
\begin{array}{rl}
&E\left[|Y_{\tau_{x,v}\wedge s}^{2;0,x,v;\varepsilon}|^{2}+\displaystyle\int_{\tau_{x,v}\wedge s}^{\tau_{x,v}\wedge\varepsilon}|Z_{r}^{2;0,x,v;\varepsilon}|^{2}dr\right]\medskip\\
=& 2E\left[\displaystyle\int_{\tau_{x,v}\wedge s}^{\tau_{x,v}\wedge\varepsilon}Y_{r}^{2;0,x,v;\varepsilon}
F(x,Y_{r}^{2;0,x,v;\varepsilon}, Z_{r}^{2;0,x,v;\varepsilon},v_{s})dr\right]
\medskip\\
\leq & 2CE\left[\displaystyle\int_{\tau_{x,v}\wedge s}^{\tau_{x,v}\wedge\varepsilon}|Y_{r}^{2;0,x,v;\varepsilon}|
\left(1+|Y_{r}^{2;0,x,v;\varepsilon}|+|Z_{r}^{2;0,x,v;\varepsilon}|\right)dr\right]\medskip\\
\leq & CE[\tau_{x,v}\wedge(\varepsilon-s)]+CE\left[\displaystyle\int_{\tau_{x,v}\wedge s}^{\tau_{x,v}\wedge\varepsilon}|Y_{r}^{2;0,x,v;\varepsilon}|^{2}dr\right]
+\frac{1}{2}E\left[\displaystyle\int_{\tau_{x,v}\wedge s}^{\tau_{x,v}\wedge\varepsilon}|Z_{r}^{2;0,x,v;\varepsilon}|^{2}dr\right]\medskip\\
\leq & C\varepsilon+CE\left[\displaystyle\int_{\tau_{x,v}\wedge s}^{\tau_{x,v}\wedge\varepsilon}|Y_{r}^{2;0,x,v;\varepsilon}|^{2}dr
+\frac{1}{2}\displaystyle\int_{\tau_{x,v}\wedge s}^{\tau_{x,v}\wedge\varepsilon}|Z_{r}^{2;0,x,v;\varepsilon}|^{2}dr\right].
\end{array}
\]
Thus, there exists a constant $C$ independent of $\varepsilon$, such that for all $s\in[0,\varepsilon]$,
\[
\begin{array}{rl}
&E\left[|Y_{\tau_{x,v}\wedge s}^{2;0,x,v;\varepsilon}|^{2}+\displaystyle\int_{\tau_{x,v}\wedge s}^{\tau_{x,v}\wedge\varepsilon}|Z_{r}^{2;0,x,v;\varepsilon}|^{2}dr
\right]
\medskip\\
\leq & C\varepsilon+CE\left[\displaystyle\int_{\tau_{x,v}\wedge s}^{\tau_{x,v}\wedge\varepsilon}|Y_{r}^{2;0,x,v;\varepsilon}|^{2}dr\right]
\leq C\varepsilon+CE\left[\displaystyle\int_{s}^{\varepsilon}|Y_{\tau_{x,v}\wedge r}^{2;0,x,v;\varepsilon}|^{2}dr\right],
\end{array}
\]
and the Gronwall inequality yields
\[
E\left[|Y_{\tau_{x,v}\wedge s}^{2;0,x,v;\varepsilon}|^{2}+\int_{\tau_{x,v}\wedge s}^{\tau_{x,v}\wedge\varepsilon}|Z_{r}^{2;0,x,v;\varepsilon}|^{2}dr\right]\leq C\varepsilon.
\]
Then, from $(\ref{BSDE Y2})$
\[
\begin{array}{rl}
E\left[\left|Y_{\tau_{x,v}\wedge s}^{2;0,x,v;\varepsilon}\right|\right]
& \leq E\left[\displaystyle\int_{\tau_{x,v}\wedge s}^{\tau_{x,v}\wedge\varepsilon}|F(x,Y_{r}^{2;0,x,v;\varepsilon}, Z_{r}^{2;0,x,v;\varepsilon},v_{r})|dr\right]\medskip\\
& \leq CE\left[\displaystyle\int_{\tau_{x,v}\wedge s}^{\tau_{x,v}\wedge\varepsilon}
\left(1+|Y_{r}^{2;0,x,v;\varepsilon}|+|Z_{r}^{2;0,x,v;\varepsilon}|\right)dr\right] \leq C\varepsilon,\quad s\in[0,\varepsilon].
\end{array}
\]
On the other hand,
we obtain from the latter estimates
\[
\begin{array}{rl}
&E\left[\displaystyle\int_{0}^{\tau_{x,v}\wedge\varepsilon}|Z_{r}^{2;0,x,v;\varepsilon}|^{2}dr\right]
=E\left[\left|\displaystyle\int_{0}^{\tau_{x,v}\wedge\varepsilon}Z_{r}^{2;0,x,v;\varepsilon}dB_{r}\right|^{2}\right]\medskip\\
\leq &2\varepsilon E\left[\displaystyle\int_{0}^{\tau_{x,v}\wedge\varepsilon}\left|F(x,Y_{r}^{2;0,x,v;\varepsilon}, Z_{r}^{2;0,x,v;\varepsilon},v_{r})\right|^{2}dr\right]+2|Y_{0}^{2;0,x,v;\varepsilon}|^{2}\medskip\\
\leq &C\varepsilon E\left[\displaystyle\int_{0}^{\tau_{x,v}\wedge\varepsilon}\left(1+|Y_{r}^{2;0,x,v;\varepsilon}|+ |Z_{r}^{2;0,x,v;\varepsilon}|\right)^{2}dr\right]+2C\varepsilon^{2}\leq C\varepsilon^{2}.
\end{array}
\]
Therefore,
\[
E\int_{0}^{\tau_{x,v}\wedge\varepsilon}\left(|Y_{s}^{2;0,x,v;\varepsilon}|+|Z_{s}^{2;0,x,v;\varepsilon}|\right)ds
\leq C\varepsilon^{2}
+\varepsilon^{\frac{1}{2}}\left\{E\left[\int_{0}^{\tau_{x,v}\wedge\varepsilon}|Z_{r}^{2;0,x,v;\varepsilon}|^{2}dr\right]\right\}^{\frac{1}{2}}
\leq C\varepsilon^{\frac{3}{2}}.
\]
\hfill
\end{proof}
\medskip

Now we define
\[
F_{0}(x,y,z):=\inf\limits_{v\in V}F(x,y,z,v).
\]
With $(H_{1})$-$(H_{6})$ we can check that $F(x,y,z,v)$ is Lipschitz continuous in $x,y,z$, uniformly w.r.t. $v$ (we denote the Lipschitz constant by $L_{0}$). Moreover, for arbitrary $v\in V$,
\[
\begin{array}{rl}
F(x,y,z,v) &\ge F(x,0,0,v)-L_{0}|y|-L_{0}|z|\medskip\\
&\ge\inf\limits_{v\in V}F(x,0,0,v)-L_{0}|y|-L_{0}|z|\medskip\\
& =F_{0}(x,0,0)-L_{0}|y|-L_{0}|z|.
\end{array}
\]
Let us consider the following BSDE
\begin{equation}\label{BSDE Y3}
\left\{
\begin{array}
[c]{l}
-dY_{s}^{3;0,x,v}=\left(F_{0}(x,0,0)-L_{0}|Y_{s}^{3;0,x,v}|-L_{0}|Z_{s}^{3;0,x,v}|\right)ds-Z_{s}^{3;0,x,v}dB_{s}, \quad 0\leq s\leq\tau_{x,v}\wedge\varepsilon, \medskip\\
Y_{\tau_{x,v}\wedge\varepsilon}^{3;0,x,v}=0.
\end{array}
\right.
\end{equation}
By setting $Y_{s}^{3;0,x,v}=0$, $Z_{s}^{3;0,x,v}=0$, for $s\in[\tau_{x,v}\wedge\varepsilon,\varepsilon]$,
we have that (\ref{BSDE Y3}) is equivalent to the following BSDE
\begin{equation}\label{BSDE Y3 2}
\left\{
\begin{array}
[c]{l}
-dY_{s}^{3;0,x,v}=1_{\{s\leq\tau_{x,v}\wedge\varepsilon\}}
\left(F_{0}(x,0,0)-L_{0}|Y_{s}^{3;0,x,v}|-L_{0}|Z_{s}^{3;0,x,v}|\right)ds
-Z_{s}^{3;0,x,v}dB_{s},  ~s\in[0,\varepsilon],\medskip\\
Y_{\varepsilon}^{3;0,x,v}=0.
\end{array}
\right.
\end{equation}
We need the following lemma
\begin{lemma}\label{lemma 3 for proving viscosity solution}
Under the assumptions $(H_{1})$-$(H_{6})$ we have
\begin{equation}\label{estimate between Y2 and Y3}
Y_{s}^{3;0,x,v}\leq Y_{s}^{2;0,x,v;\varepsilon}, \quad \text{ for all } s\in[0,\tau_{x,v}\wedge\varepsilon],~v\in\mathcal{V}, ~\mathbb{P}\text{-}a.s.
\end{equation}
Moreover, for $x\in \overline{D}\setminus\partial D$, there exists a constant $C$ independent of $\varepsilon$ and $v$ such that
\begin{equation}\label{estimate between Y3 and Y4}
|Y_{0}^{3;0,x,v}-Y_{0}^{4;0,x}|\leq C\varepsilon^{\frac{3}{2}},
\end{equation}
where $Y_{s}^{4;0,x}$ is the solution of the following ordinary differential equation
\begin{equation}\label{ODE}
\left\{
\begin{array}
[c]{l}
-dY_{s}^{4;0,x}=\left(F_{0}(x,0,0)-L_{0}|Y_{s}^{4;0,x}|\right)ds,\quad s\in[0,\varepsilon],\medskip\\
Y_{\varepsilon}^{4;0,x}=0.
\end{array}
\right.
\end{equation}
\end{lemma}
\begin{proof}
Comparing (\ref{BSDE Y2}) and (\ref{BSDE Y3}) and using
$F_{0}(x,0,0)-L_{0}|y|-L_{0}|z|\leq F(x,y,z,v)$, for all $v\in V$,
Lemma \ref{comparison theorem} yields (\ref{estimate between Y2 and Y3}). To complete the proof, it remains to show (\ref{estimate between Y3 and Y4}).

First, one can check that the solution of (\ref{ODE}) is given by
\begin{equation}\label{solution of ODE}
Y_{s}^{4;0,x}=\left\{
\begin{array}
[c]{rl}
&\frac{1}{L_{0}}F_{0}(x,0,0)(1-e^{-L_{0}(\varepsilon-s)}),\quad F_0(x,0,0)\ge0, \quad s\in[0,\varepsilon],\medskip\\
&\frac{1}{L_{0}}F_{0}(x,0,0)(e^{L_{0}(\varepsilon-s)}-1),\quad F_0(x,0,0)<0, \quad s\in[0,\varepsilon].
\end{array}
\right.
\end{equation}
Obviously, $|Y_{s}^{4;0,x}|\leq C(\varepsilon-s)\leq C\varepsilon$,  $s\in[0,\varepsilon]$,
and
\begin{equation}\label{estimate between F0 and Y4}
\left|F_{0}(x,0,0)-L_{0}|Y_{s}^{4;0,x}|\right|=\left|F_{0}(x,0,0)e^{L_{0}(\varepsilon-s)}\right|\leq Ce^{L_{0}(\varepsilon-s)}\leq C,\quad s\in[0,\varepsilon].
\end{equation}
By applying It\^{o}'s formula to $|Y_{s}^{3;0,x,v}-Y_{s}^{4;0,x}|^{2}$, we deduce from (\ref{BSDE Y3 2}) and (\ref{ODE}), that
\[
\begin{array}{rl}
&E\left[|Y_{s}^{3;0,x,v}-Y_{s}^{4;0,x}|^{2}+\displaystyle\int_{s}^{\varepsilon}|Z_{r}^{3;0,x,v}|^{2}dr|\mathcal{F}_s\right]\medskip\\
=& -2L_{0}E\left[\displaystyle\int_{s}^{\tau_{x,v}\wedge\varepsilon}(Y_{r}^{3;0,x,v}-Y_{r}^{4;0,x})
(|Y_{r}^{3;0,x,v}|-|Y_{r}^{4;0,x}|+|Z_{r}^{3;0,x,v}|)dr|\mathcal{F}_s\right]\medskip\\
&+2E\left[\displaystyle\int_{\tau_{x,v}\wedge\varepsilon}^{\varepsilon}Y_{r}^{4;0,x}(F_{0}(x,0,0)-L_{0}|Y_{r}^{4;0,x}|)dr|\mathcal{F}_s\right]
\medskip\\
\leq & 2(L_{0}+L_{0}^{2})E\left[\displaystyle\int_{s}^{\varepsilon}|Y_{r}^{3;0,x,v}-Y_{r}^{4;0,x}|^{2}dr|\mathcal{F}_s\right]
+\frac{1}{2}E\left[\displaystyle\int_{s}^{\varepsilon}|Z_{r}^{3;0,x,v}|^{2}dr|\mathcal{F}_s\right]+2C^{2}\varepsilon^{2}.
\end{array}
\]
Thus, there exists a constant $C$ independent of $\varepsilon$ and $v$, such that
\[
\begin{array}{rl}
&E\left[|Y_{s}^{3;0,x,v}-Y_{s}^{4;0,x}|^{2}|\mathcal{F}_s\right]
+\frac{1}{2}E\left[\displaystyle\int_{s}^{\varepsilon}|Z_{r}^{3;0,x,v}|^{2}dr|\mathcal{F}_s\right]
\medskip\\
\leq & 2(L_{0}+L_{0}^{2})E\left[\displaystyle\int_{s}^{\varepsilon}|Y_{r}^{3;0,x,v}-Y_{r}^{4;0,x}|^{2}dr|\mathcal{F}_s\right]
+C\varepsilon^{2},
\end{array}
\]
and the Gronwall inequality yields
\[
E\left[|Y_{s}^{3;0,x,v}-Y_{s}^{4;0,x}|^{2}+\int_{s}^{\varepsilon}|Z_{r}^{3;0,x,v}|^{2}dr|\mathcal{F}_s\right]\leq C\varepsilon^{2},\quad s\in[0,\varepsilon].
\]
Consequently, $|Y_{s}^{3;0,x,v}-Y_{s}^{4;0,x}|\leq C\varepsilon$, $s\in[0,\varepsilon]$, and
\[
E\left[\int_{0}^{\varepsilon}|Z_{r}^{3;0,x,v}|^{2}dr\right]\leq C\varepsilon^{2}.
\]
Using the equations (\ref{BSDE Y3 2}) and (\ref{ODE}) again, and recalling (\ref{estimate between F0 and Y4}), we have
\[
\begin{array}{ll}
\quad |Y_{0}^{3;0,x,v}-Y_{0}^{4;0,x}|=E\left[|Y_{0}^{3;0,x,v}-Y_{0}^{4;0,x}|\right]\medskip\\
\leq L_{0}E\left[\displaystyle\int_{0}^{\tau_{x,v}\wedge\varepsilon}(|Y_{r}^{3;0,x,v}-Y_{r}^{4;0,x}|+|Z_{r}^{3;0,x,v}|)dr\right]
+E\left[\displaystyle\int_{\tau_{x,v}\wedge\varepsilon}^{\varepsilon}|F_{0}(x,0,0)-L_{0}|Y_{r}^{4;0,x}||dr\right]
\medskip\\
\leq C\varepsilon E\left[\sup\limits_{s\in[0,\varepsilon]}|Y_{s}^{3;0,x,v}-Y_{s}^{4;0,x}|\right]
+C\varepsilon^{\frac{1}{2}}\left\{E\left[\displaystyle\int_{0}^{\varepsilon}|Z_{r}^{3;0,x,v}|^{2}dr\right]\right\}^{1/2}
+CE[\varepsilon-\varepsilon\wedge\tau_{x,v}]\medskip\\
\leq C\varepsilon^{\frac{3}{2}}+C\varepsilon E[1_{\{\tau_{x,v}\leq \varepsilon\}}].
\end{array}
\]
Noticing that for $x\in \overline{D}\setminus\partial D$ we can assume that there exist a $\delta_{0}>0$, such that $dist(x,\partial D)\ge\delta_{0}>0$, then from (\ref{estimate for SDE}), we have, uniformly in $v\in\mathcal{V}$,
\[
E[1_{\{\tau_{x,v}\leq \varepsilon\}}]=\mathbb{P}(\tau_{x,v}\leq \varepsilon)\leq \mathbb{P}(\sup\limits_{s\in[0,\varepsilon]}|X_{s}^{0,x,v}-x|\ge\delta_{0})\leq \frac{1}{|\delta_{0}|^{4}}E\sup\limits_{s\in[0,\varepsilon]}|X_{s}^{0,x,v}-x|^{4}\leq C\varepsilon^{2}.
\]
Consequently,
$|Y_{0}^{3;0,x,v}-Y_{0}^{4;0,x}|\leq C\varepsilon^{\frac{3}{2}}$.
\end{proof}
\begin{remark}
For $x\in \partial D$, we don't have (\ref{estimate between Y3 and Y4}). Indeed, as proved in Lemma \ref{lemma for assumption}, under assumption $(H_{4})$,
\[
\partial D=\Gamma:=\left\{x\in\partial D: \mathbb{P}(\tau_{x,v}>0)=0\right\}, \quad \text{ for all } v\in\mathcal{V}.
\]
Consequently, $\tau_{x,v}=0$, $(Y^{3;0,x,v},Z^{3;0,x,v})=(0,0)$, $s\in[0,\varepsilon]$, while $Y_{s}^{4;0,x}=\frac{1}{L_{0}}F_{0}(x,0,0)(1-e^{-L_{0}(\varepsilon-s)})$, if $F_0(x,0,0)\ge0$ and $Y_{s}^{4;0,x}=\frac{1}{L_{0}}F_{0}(x,0,0)(e^{L_{0}(\varepsilon-s)}-1)$, if $F_0(x,0,0)<0$, $s\in[0,\varepsilon]$ (see (\ref{solution of ODE})).
\end{remark}

Now we can give one of the main results of this section.
\begin{theorem}\label{theorem supersolution}
We suppose that the assumptions $(H_{1})$-$(H_{6})$ are satisfied. We also assume that $g\in W^{2,\infty}(D)$ and there exists a constant $\theta$ such that $\beta^{2}-2\alpha<\theta\leq\mu$ and $\theta< -2[\delta]^{+}$. Then the value function defined by (\ref{value function}) is a viscosity supersolution of (\ref{HJB equation}).
\end{theorem}
\begin{proof}
Let us first check that $u(x)\ge g(x)$, for $x\in\partial D$. Indeed, we have $u(x)=g(x)$. This is because for $x\in\partial D$, from above remark, we have $\tau_{x,\upsilon}=0$, for all $v\in\mathcal{V}$. Then, from the definition of the value function and the solution of the BSDE (\ref{BSDE coupled with SDE}), we have $u(x)=g(x)$.

Now we suppose that $\varphi\in C^{3}(\overline{D})$ and $u-\varphi$ achieves a local minimum (w.l.o.g. we can assume it to be a global one) at $x\in \overline{D}\setminus\partial D$. Then we have $\tau_{x,v}>0$, a.s. We may also suppose that $u(x)=\varphi(x)$, and hence $u(\bar{x})\ge\varphi(\bar{x})$, for all $\bar{x}\in\overline{ D}$.
Then, given an arbitrary $\varepsilon>0$, by the dynamic programming principle (see Theorem \ref{DPP}) it holds
\[
\varphi(x)=u(x)=\inf\limits_{v\in\mathcal{V}}G_{\tau_{x,v}\wedge\varepsilon}^{0,x,v}
[u(X_{\tau_{x,v}\wedge\varepsilon}^{0,x,v})],
\]
and from the comparison theorem for BSDEs (see Lemma \ref{comparison theorem}) and $u\ge\varphi$ on $\overline{D}$ we have
\[
\inf\limits_{v\in\mathcal{V}}\left(G_{\tau_{x,v}\wedge\varepsilon}^{0,x,v}
[\varphi(X_{\tau_{x,v}\wedge\varepsilon}^{0,x,v})]-\varphi(x)\right)
\leq\inf\limits_{v\in\mathcal{V}}G_{\tau_{x,v}\wedge\varepsilon}^{0,x,v}
[u(X_{\tau_{x,v}\wedge\varepsilon}^{0,x,v})]-\varphi(x)
=0.
\]
Hence, from Lemma \ref{lemma 1 for proving viscosity solution}, it follows that
$\inf\limits_{v\in\mathcal{V}}Y_{0}^{1;0,x,v;\varepsilon}\leq 0$,
and we can find $\widetilde{v}(\cdot)\in\mathcal{V}$ depending on $\varepsilon$ such that $Y_{0}^{1;0,x,\widetilde{v};\varepsilon}\leq \varepsilon^{\frac{3}{2}}$.
Thus, from the Lemmas \ref{lemma 2 for proving viscosity solution} and \ref{lemma 3 for proving viscosity solution} (\ref{estimate between Y2 and Y3}) we obtain
\[
Y_{0}^{3;0,x,\widetilde{v}}\leq Y_{0}^{2;0,x,\widetilde{v};\varepsilon}\leq C \varepsilon^{\frac{3}{2}},
\]
and Lemma \ref{lemma 3 for proving viscosity solution} (\ref{estimate between Y3 and Y4}) yields that
$Y_{0}^{4;0,x}\leq 2C\varepsilon^{\frac{3}{2}}$.
Using the explicit expression (\ref{solution of ODE}) for $Y_{0}^{4;0,x}$, we obtain
\[
\frac{1}{L_{0}}F_{0}(x,0,0)(1-e^{-L_{0}\varepsilon})\leq C\varepsilon^{\frac{3}{2}},   \text{ if } F_{0}(x,0,0)\ge0,
\]
and
\[
\frac{1}{L_{0}}F_{0}(x,0,0)(e^{L_{0}\varepsilon}-1)\leq C\varepsilon^{\frac{3}{2}},  \text{ if }  F_{0}(x,0,0)<0.
\]
Consequently, dividing both sides by $\varepsilon$, and taking the limit $\varepsilon\searrow0$, we get always
\[
F_{0}(x,0,0)=\inf\limits_{v\in V}F(x,0,0,v)\leq 0.
\]
Recalling the definition of $F$, we see that the latter relation is nothing else than
\[
\inf\limits_{v\in V}\left\{\mathcal{L}(x,v)\varphi(x)+f(x,u(x),\nabla \varphi(x) \sigma(x,v),v)\right\}\leq 0,
\quad x\in \overline{D}\setminus\partial D.
\]
We complete the proof.
\end{proof}

Now we are going to show that $u$ is a viscosity subsolution.
\begin{theorem}\label{theorem subsolution}
Under the assumptions of Theorem \ref{theorem supersolution}, the value function defined by (\ref{value function}) is a viscosity subsolution of (\ref{HJB equation}).
\end{theorem}
\begin{proof}
For $x\in\partial D$, we have $u(x)=g(x)$.
We suppose that $\varphi\in C^{3}(\overline{D})$ and $u-\varphi$ achieves a global maximum at $x\in \overline{D}\setminus\partial D$. Then we have $\tau_{x,v}>0$, a.s. As before, we may also suppose that $u(x)=\varphi(x)$. Hence $u(\bar{x})\leq\varphi(\bar{x})$, for all $\bar{x}\in\overline{ D}$.
We have to prove that $\inf\limits_{v\in V}F(x,0,0,v)\ge 0$. Let us suppose that it's not true, i.e. there exists some constant $m_{0}$ s.t.
\begin{equation}\label{contradiction assumption subsolution}
\inf\limits_{v\in V}F(x,0,0,v)\leq -m_{0}<0.
\end{equation}
The dynamic programming principle (see Theorem \ref{DPP}) implies that
\[
\varphi(x)=u(x)=\inf\limits_{v\in\mathcal{V}}G_{\tau_{x,v}\wedge\varepsilon}^{0,x,v}
[u(X_{\tau_{x,v}\wedge\varepsilon}^{0,x,v})].
\]
Then, from the comparison theorem for BSDEs (see Lemma \ref{comparison theorem}) and $u\leq\varphi$ it follows that
\[
\inf\limits_{v\in\mathcal{V}}\left(G_{\tau_{x,v}\wedge\varepsilon}^{0,x,v}
[\varphi(X_{\tau_{x,v}\wedge\varepsilon}^{0,x,v})]-\varphi(x)\right)
\ge\inf\limits_{v\in\mathcal{V}}G_{\tau_{x,v}\wedge\varepsilon}^{0,x,v}
[u(X_{\tau_{x,v}\wedge\varepsilon}^{0,x,v})]-\varphi(x)
=0,
\]
and from Lemma \ref{lemma 1 for proving viscosity solution} we have
$Y_{0}^{1;0,x,\bar{v};\varepsilon}\ge\inf\limits_{v\in\mathcal{V}}Y_{0}^{1;0,x,v;\varepsilon}\ge 0$,
where $\bar{v}\in V$ is such that $F(x,0,0,\bar{v})=F_{0}(x,0,0)=\inf\limits_{v\in V}F(x,0,0,v)$.
From Lemma \ref{lemma 2 for proving viscosity solution} (\ref{estimate between Y1 and Y2}), we obtain
\begin{equation}\label{estimate 1 for contradiction subsolution}
Y_{0}^{2;0,x,\bar{v};\varepsilon}\ge-C \varepsilon^{\frac{3}{2}}.
\end{equation}
Taking into account that
\[
Y_{0}^{2;0,x,\bar{v};\varepsilon}=E\left[\int_{0}^{\tau_{x,\bar{v}}\wedge\varepsilon}F(x,Y_{s}^{2;0,x,\bar{v};\varepsilon}, Z_{s}^{2;0,x,\bar{v};\varepsilon},\bar{v})ds\right],
\]
we get from the Lipschitz continuity of $F$ in $(y,z)$, (\ref{contradiction assumption subsolution}) as well as Lemma \ref{lemma 2 for proving viscosity solution} (\ref{estimate for Y2})
\begin{equation}\label{estimate 2 for contradiction subsolution}
\begin{array}{rl}
Y_{0}^{2;0,x,\bar{v};\varepsilon}
& \leq E\left[\displaystyle\int_{0}^{\tau_{x,\bar{v}}\wedge\varepsilon}\left(F(x,0, 0,\bar{v})+C|Y_{s}^{2;0,x,\bar{v};\varepsilon}|+C|Z_{s}^{2;0,x,\bar{v};\varepsilon}|\right)ds\right]\medskip\\
&\leq -m_{0}E\left[\tau_{x,\bar{v}}\wedge\varepsilon\right]+C\varepsilon^{\frac{3}{2}}
\leq -m_{0}\varepsilon\mathbb{P}(\tau_{x,\bar{v}}>\varepsilon)+C\varepsilon^{\frac{3}{2}}.
\end{array}
\end{equation}
Comparing (\ref{estimate 1 for contradiction subsolution}) and (\ref{estimate 2 for contradiction subsolution}), we have
$-C\varepsilon^{\frac{3}{2}}\leq -m_{0}\varepsilon\mathbb{P}(\tau_{x,\bar{v}}>\varepsilon)+C\varepsilon^{\frac{3}{2}}$,
which implies that $-2C\varepsilon^{\frac{1}{2}}\leq -m_{0}\mathbb{P}(\tau_{x,\bar{v}}>\varepsilon)$.
Taking the limit as $\varepsilon\searrow 0$, we have $0\leq -m_{0}\mathbb{P}(\tau_{x,\bar{v}}>0)=-m_{0}$ (Recall that $x\in\overline{D}\setminus \partial D$). But this means $m_{0}\leq 0$, which is in contradiction to (\ref{contradiction assumption subsolution}).
\end{proof}
\smallskip

Combining Theorems \ref{theorem supersolution} and \ref{theorem subsolution} we have
\begin{theorem}\label{theorem solution}
We suppose that the assumptions $(H_{1})$-$(H_{6})$ are satisfied. We also assume that $g\in W^{2,\infty}(D)$ and there exists a constant $\theta$ such that $\beta^{2}-2\alpha<\theta\leq\mu$ and $\theta< -2[\delta]^{+}$. Then the value function defined by (\ref{value function}) is a viscosity solution of (\ref{HJB equation}).
\end{theorem}
Finally, we also have the uniqueness of the viscosity solution of HJB equation (\ref{HJB equation}) in the class of $1/2$-H\"{o}lder continuous functions on $\overline{D}$.
\begin{theorem}\label{uniqueness of HJB equation}
We suppose that the assumptions $(H_{1})$-$(H_{6})$ are satisfied. Then HJB equation (\ref{HJB equation}) has at most one viscosity solution in the class of $1/2$-H\"{o}lder continuous functions on $\overline{D}$.
\end{theorem}
\begin{proof}
To prove the theorem, it is sufficient to show that if $u_{1}$ (resp. $u_{2}$) is a $1/2$-H\"{o}lder continuous subsolution (resp. supersulotion), then $u_{1}\leq u_{2}$ for all $x\in \overline{D}$. One can check easily that under assumptions $(H_{1})$-$(H_{6})$,
\[
\overline{F}:=-\inf\limits_{v\in V}\left\{\mathcal{L}(x,v)u(x)+f(x,u(x),\nabla u(x) \sigma(x,v),v)\right\}
\]
satisfies the assumptions of Theorem 3.3 \cite{CIL-92}. The proof is complete.
\end{proof}

\end{document}